\documentclass{article}
\usepackage{amsmath,amssymb,amsthm}
\usepackage{graphicx}
\usepackage{float}
\def\Oh{\mathcal{O}}
\def\eps{\varepsilon}
\def\be{\begin{equation}}
\def\ee{\end{equation}}
\def\dss{\displaystyle}

\newtheorem{thm}{Theorem}
\newtheorem{lemma}{Lemma}
\newtheorem{remark}{Remark}

\title{Improving the Accuracy of the Exponentially Fitted Scheme on Piecewise Uniform Meshes}
\author{Relja Vulanovi\'c\footnote{Professor Emeritus, rvulanov@kent.edu}\\
Department of Mathematical Sciences\\ 
Kent State University at Stark\\ 
6000 Frank Ave NW\\ 
North Canton, OH 44720, USA 
}
\begin{document}
\maketitle

\begin{abstract}
\noindent
A linear one-dimensional singularly perturbed convection-diffusion problem is solved numerically after its solution is decomposed as $u_0+w$, where
$u_0$, the corresponding reduced solution, is treated as a function known exactly or approximately. The component $w$ is then calculated 
using the exponentially fitted Allen-Southwell-Il'in (ASI) scheme on the Shishkin mesh and its asymptotic version. We prove that this numerical method is 
highly accurate, with errors that diminish when the discretization parameter increases, and, in some cases, even when the perturbation parameter decreases. 
This is a theoretical confirmation of earlier numerical results showing that the ASI scheme outperforms the general class of Samarskii-type schemes to which 
it belongs. Even higher accuracy is proved when $u_0$ is linear, in which case, the decomposition is not needed. New numerical experiments are provided to illustrate all this.
\\[5mm]
Keywords: singular perturbation, convection-diffusion, solution decomposition, Shishkin mesh, asymptotic mesh,
exponential fitting, Samarskii-type schemes
\end{abstract}

\section{Introduction}
\label{sec:intro}

We consider the problem of finding a $C^2[0,1]$-solution to the following linear one-dimensional singularly perturbed convection-diffusion problem,
\begin{equation}\label{1DCD}
Lu:= -\varepsilon u'' -b(x)u' + c(x)u=f(x),\ x\in(0,1),\
u(0)=u(1)=0.
\end{equation}
The parameter $\eps$ is a perturbation parameter, $0< \eps \ll 1$, and 
the functions $b$, $c$, and $f$ are sufficiently smooth. Moreover, we assume that
\be\label{bc_cond}
b(x)> \beta>0, \ \ c(x)\ge 0, \ \ x\in I:=[0,1].
\ee
Then, \eqref{1DCD} has a unique solution $u$ that in general shows an exponential layer near the boundary at $x=0$, \cite{KT78}.

We are interested in numerically solving the problem \eqref{1DCD}. This interest is motivated by the fact that singular 
perturbation problems arise in various applications, see \cite{ChH, FHMOS, RST08} for instance, and the problem 
\eqref{1DCD} is one of the most frequently analyzed models.

Singularly perturbed problems require numerical methods that produce approximate solutions with errors that decrease 
when the discretization parameter $N$ grows, but do not increase when $\eps\to 0$. This can be achieved on a uniform 
mesh by adjusting the discretization scheme to the behavior of the solution $u$. For the problem \eqref{1DCD}, a 
scheme of this kind is the exponentially fitted scheme due to Allen and Southwell \cite{AllSouth}, and Il'in \cite{Il69}. We 
refer to this scheme as the {\em Allen-Southwell-Il'in}\/ ({\em ASI}) scheme. 
Early detailed analyses of the ASI scheme can be found in \cite{Il69, KT78, DMS80}. 
They demonstrate that the ASI scheme is $\eps$-uniformly first-order accurate. 
The upwind scheme and the Samarskii scheme, which are also analyzed in \cite{KT78}, do not have this property, 
although they are, respectively, first- and second-order accurate when $\eps$ is fixed. A general class of 
schemes, including the ASI scheme and the Samarskii scheme, is considered in \cite{Tob83}. We call this class
the {\em Samarskii-type}\/ schemes. The Samarskii-type schemes that are not exponentially fitted can only achieve
first-order $\eps$-uniform accuracy away from the layer, that is, in an interval of the form $[\delta,1]$ with a fixed 
positive $\delta$, \cite{Tob83}.

Another way to get $\eps$-uniform numerical results is to use special nonuniform discretization meshes dense in the 
layer, such as the Bakhvalov \cite{Bakh69} and Shishkin \cite{Sh88} meshes. A natural question is whether the 
advantage the ASI scheme has over other Samarskii-type schemes on the uniform mesh carries over to the
layer-adapted meshes. The results in \cite{Linss02} and \cite{VNfc} indicate that this is not the case on the
Shishkin mesh. Based on our numerical experiments, the same can be said about the Bakhvalov mesh. 
On the Shishkin mesh, all these schemes have $\eps$-uniform 
accuracy of order $\Oh\left(N^{-1}\right)$, \cite{Linss02}. This is an improvement over other first-order schemes, 
represented by the upwind scheme, whose error behaves like $\Oh\left(N^{-1}\ln N \right)$ uniformly in $\eps$. 
The same result as in \cite{Linss02} is reiterated in \cite{VNfc}, but it is also demonstrated there that the Samarskii-type 
schemes are almost-second-order accurate, with an error of $\Oh\left( N^{-2} (\ln N)^2\right)$, on the layer 
component of the solution. Based on this property, the accuracy is improved to almost-second-order using the solution 
decomposition $u=u_0+w$, where $u_0$ is the solution of the reduced problem corresponding to \eqref{1DCD}.

The numerical results in \cite{VNfc} demonstrate that the above-described use of $u_0$ enables the ASI scheme to 
regain its superiority over other Samarskii-type schemes. Our goal is to analyze this in the present paper. In addition to 
the standard Shishkin mesh considered in \cite{VNfc}, we also discuss the closely related asymptotic mesh ($A$-mesh) 
\cite[p.~405]{RST08}. In the original Shishkin mesh, the transition point between the fine part of the mesh in
the layer and the coarse part outside the layer is of the form $\Oh(\eps \ln N)$. The $A$-mesh we consider has 
$\Oh(\eps|\ln\eps|)$ for the transition point. Using two proof techniques, viz., the barrier-function 
technique and the preconditioning technique, we prove that the ASI scheme is highly accurate, with the error that in 
some cases decreases when $\eps$ does. 

This paper is a continuation not only of \cite{VNfc}, but also of \cite{VN21}, where a method similar to the present one
is analyzed on the $A$-mesh. However, the theoretical results derived here cannot be directly deduced 
from either \cite{VNfc} or \cite{VN21}. In \cite{VNfc}, only the general class of Samarskii-type schemes is discussed, and 
the specific features of the ASI scheme are not considered.  As for \cite{VN21}, it only deals with the simpler 
upwind scheme, but the method is based on the decomposition $u=u_0 + v + y$, where $v$ is the exponential-layer 
function. The present approach is simpler and produces more accurate results.

In the next section, we introduce some preliminaries: the solution decomposition and its analysis, the joint definition of the two piecewise uniform meshes, and the family of 
the Samarskii-type schemes. The numerical method is described in Section \ref{sec:nm}. This is followed by the discussion of the two proof techniques in
Section \ref{sec:bar-pre}. Section \ref{sec:trunc} deals with estimates of the truncation error. Then, the error estimates for the numerical method are derived in
Section \ref{sec:main}. Two further issues are discussed in the subsections of Section \ref{sec:disc}: the situation when the reduced solution is found numerically,
and a special case when the reduced solution is linear, which is when the decomposition $u=u_0+w$ is not needed. Numerical results are provided in
Section \ref{sec:numer} and some concluding remarks in Section \ref{sec:concl}.

 \section{Preliminaries}
 \label{sec:prel}
 
Throughout the paper, $C$ denotes a generic positive constant independent of $\eps$ and $N$. Some specific constants of this kind are sub- or superscripted.
We use the continuous maximum norm $\| \cdot\|_\Omega$, where $\Omega$ is a closed subset of $I$.
 
We next describe a decomposition of the solution $u$ to the problem \eqref{1DCD}, which is vital for our numerical method. 
Then, we introduce the two kinds of piecewise-uniform discretization meshes and Samarskii-type schemes. 

\subsection{The Solution Decomposition and Derivative Estimates}
\label{sec:decomp}

We now specify the smoothness requirement for the problem \eqref{1DCD}: $b, c, f\in C^4(I)$. Then the solution $u$ of \eqref{1DCD} belongs 
to $C^6(I)$ and the corresponding reduced solution $u_0$ is in $C^5(I)$. 
The reduced solution solves the reduced problem, which is a terminal-value problem,
\be
\label{red}
-b(x)u' + c(x)u = f(x), \ \ x\in (0,1), \ \ u(1)=0.
\ee
The following decomposition of the solution $u$ is analyzed in \cite{VN21}:
\[
u = u_0 + w, \ \ w = y + v, \ \  v(x) = -\frac{\eps u'(0)}{b(0)}e^{-b(0)x/\eps}.
\]
This is a modification of the Kellogg-Tsan decomposition \cite{KT78}, which is $u = z + v$. It is proved in \cite{VN21} that
\be
\label{dery}
	|y^{(j)}(x)| \le C\left( \eps + \eps^{1-j}e^{-\beta x/\eps}\right), \ \ j=0,1,2,3,4, \ \ x\in [0,1].
\ee
Since $\eps |u'(0)| \le C$, we have
\[
	|v^{(j)}(x)| \le C \eps^{-j}e^{-b(0) x/\eps}, \ \ j=0,1,2,3,4, \ \ x\in [0,1].
\]

The decomposition  $u = u_0 + w$ is the basis of the numerical method proposed in \cite{VNfc} for 
solving the problem \eqref{1DCD}. The problem is transformed into
\be
\label{wprob}
L w(x) = \eps u_0''(x), \ \ x\in (0,1), \ \ w(0) = -u_0(0), \ \ w(1)=0,
\ee
and then, \eqref{wprob} is discretized using the general Samarskii-type scheme on the standard Shishkin mesh, which is described below along with its variation.

\subsection{The Meshes} 
\label{sec:smesh}

The discretization mesh for the problem \eqref{wprob} consists of the points $x_0=0<x_1 < x_2< \cdots < x_{N-1} < x_N =1$. 
Let $h_i=x_i-x_{i-1}$ for $i=1,2,\ldots, N$ and $\hbar_i = (h_i+h_{i+1})/2$ for $i=1,2,\ldots, N-1$.

We consider two kinds of piecewise-uniform meshes; their joint description follows.
Let either $\lambda=\eps^{-1}$ or $\lambda=N$. Let also $Q$ be a fixed rational number in $(0,1)$, such that $J=QN$ is an integer, and let 
$a$ be a positive mesh parameter to be determined. Then we define 
\[
\xi = \min\left\{ Q,\ \frac{a\eps}{\beta} \ln\lambda \right\}
\]
as the transition point between the fine and coarse parts of the mesh. In what follows, we only consider the case $\xi = (a\eps/\beta)\ln\lambda$. 
Due to the smallness of $\eps$, this is not a restriction when $\lambda = \eps^{-1}$. On the other hand, if $\lambda= N$, then $\xi=Q$ 
means that $N$ is unrealistically large (the analysis is simpler in that case anyway). 

The interval $[0,\xi]$ is divided into $J$ uniform subintervals. This produces the fine mesh in the layer. The coarse part of the mesh contains the points obtained 
when the interval $[\xi,1]$ is divided into $N-J$ uniform subintervals. Therefore, the mesh steps are
\[
h_i = \left\{\begin{array}{lc}
h=h(\lambda) := \xi/J \le C \eps N^{-1}|\ln\lambda|, & 1 \le i \le J, \\
H=H(\lambda) := (1-\xi)/(N-J) \le CN^{-1}, & J+1 \le i \le N.
\end{array}\right.
\]
The mesh points
\[ 
x_i = \left\{\begin{array}{lc}
ih, & 0\le i\le J, \\
\xi + (i-J)H, & J+1 \le i \le N.
\end{array}\right.
\]
constitute the mesh, which we denote by $S_\lambda$. The mesh $S_N$ is the standard Shishkin mesh and
$S_{1/\eps}$ is an $A$-mesh.

\subsection{The Samarskii-Type Schemes}
\label{sec:Sam}

Let $U^N$, $W^N$, etc., denote mesh functions on $S_\lambda$ and let $g^N$ be the discretization of
any $C(I)$-function $g$ on $S_\lambda$. We use the discrete maximum norm, $\| U^N\| = \max_{0\le i\le N} |U^N_i|$. 
We also use the following standard finite-difference operators on $S_\lambda\setminus\{0,1\}$:
\[
D^+ U^N_i = \frac{U^N_{i+1}-U^N_i}{h_{i+1}}, \ \ D^- U^N_i = \frac{U^N_{i}-U^N_{i-1}}{h_{i}},
\]
\[
D''U^N_i=\frac{1}{\hbar_i}\left( D^+U^N_i - D^-U^N_{i} \right).
\]

The Samarskii-type schemes that discretize the continuous operator $L$ can be represented as modifications of the upwind scheme, \cite{VNfc},
\[
L^N U^N_i  :=  -\eps\sigma(\rho_i) D'' U^N_i -b(x_i) D^+ U^N_i + c(x_i) U^N_i,
\]
where
\[ 
\rho_i = \frac{b(x_i) h_{i+1}}{2\eps}
\]
and for the function $\sigma$, defined on $[0,\infty)$, it holds that
\be
\label{sigma1+}
0\le \sigma(\rho)\le \tilde C, \ \ \rho\in[0,\infty),
\ee
and
\be
\label{sigma2}
|\sigma(\rho) + \rho - 1| \le \bar C\min\left\{ \rho, \rho^2\right\}, \ \ \rho\in[0,\infty).
\ee
These conditions are also considered in \cite{Tob83}. Because of $\sigma\ge 0$, the operator $L^N$ satisfies the discrete maximum principle.

Some notable members of the class of Samarskii-type schemes are:
\begin{itemize}
\item the Samarskii scheme \cite{Sa} with 
\[ \sigma(\rho) = \sigma_{\text Sam}(\rho) = \frac{1}{1+\rho}, \]
\item the Runchal-Spalding \cite{Runch, Spald} with 
\[ \sigma(\rho) = \max\{ 1-\rho, 0\}, \]
\item the Allen-Southwell-Il'in (ASI) scheme \cite{AllSouth, Il69} with 
\[ \sigma(\rho) = \sigma_{\text{ASI}} := \rho(\coth\rho -1) = \frac{2\rho}{e^{2\rho} -1}, \]
where $\sigma_{\text{ASI}}(0)$ is defined as 1, the value of its limit when $\rho\to 0^+$.
\end{itemize}

For a more complete list of specific Samarskii-type schemes and their generalizations that satisfy the conditions \eqref{sigma1+} and \eqref{sigma2}, see \cite{VNfc}. 
The ASI scheme is not discussed 
there in detail, so we do this briefly here using $\sigma_{\text{ASI}} = 2\rho/\left( e^{2\rho}-1\right)$. The inequalities in \eqref{sigma1+} are trivially satisfied, the right 
one with $\tilde C =1$. As for \eqref{sigma2}, it holds with $\bar C=1$. We first observe that 
\be
\label{ASI}
\sigma_{\text{ASI}}(\rho) + \rho - 1 \ge 0
\ee
This is immediately true for $\rho\ge 1$. When $0\le \rho <1$, it can be proved that $e^{2\rho} \le (1+\rho)/(1-\rho)$, which is
equivalent to \eqref{ASI}. Therefore, we need to prove that
\[
\frac{2\rho}{e^{2\rho} - 1} \le 1 - \rho + \rho^s, \ \  \rho\ge 0, \ \  s =1,2.
\]
This follows from
\[
\frac{2\rho}{e^{2\rho} - 1} \le \frac{2\rho}{2\rho + 2\rho^2} = \sigma_{\text{Sam}}(\rho)
\]
since it can be directly verified that $\sigma_{\text{Sam}}(\rho) \le 1 - \rho + \rho^s$, $\rho\ge 0$, $s =1,2$.

Note that we get the upwind scheme when $\sigma(\rho)\equiv 1$, but the condition \eqref{sigma2} is not satisfied. The central scheme, 
obtained by setting $\sigma(\rho)=1-\rho$, violates \eqref{sigma1+}.

Let us mention that $L^N$ with $\sigma=\sigma_{\text{ASI}}$ is not the only possible nonuniform-mesh generalization 
of the original ASI scheme. The original scheme is constructed on a uniform mesh (which we denote by $I^N$), so that 
it is exact for the component $v$ of the solution when $b$ is constant. Since $-\eps v'' -b(x)v'=0$ on $(0,1)$ for 
$b(x) \equiv b = \text{const.}$, the fitting coefficient $\sigma_i$ is determined so that
\be
\label{vexact}
-\eps\sigma_i D''v(x_i) - b D^+ v(x_i) = 0 
\ee
on $I^N\setminus\{0,1\}$. Then, the constant $b$ is replaced with $b(x_i)$ everywhere in the scheme to form a discretization of \eqref{1DCD}. This gives
\[
\sigma_i = \frac{2\rho_i}{e^{2\rho_i} -1}, \ \ \rho_i = \frac{b(x_i)\chi}{2\eps},
\]
where $\chi$ is the step of the uniform mesh. In the nonuniform-mesh generalization we consider here, $\rho_i$ is formed simply by replacing $\chi$ with $h_{i+1}$. 
However, this construction does not preserve the property \eqref{vexact}. Consequently, the part of the error that is created by $v$ requires an analysis at the 
transition point $x_J$ that is different from the analysis at all other mesh points. For 
discretization meshes that are more nonuniform (such as the Shishkin-type meshes with a nonuniform part in the layer, \cite{RL99, Linss10}, 
or the Bakhvalov-type meshes, \cite{Bakh69, Vul83}), it may be important to have the property \eqref{vexact}. This can be achieved by defining $\sigma_i$ so that 
\eqref{vexact} holds on the nonuniform mesh. When the function $b(x)$ is not a constant, $\sigma_i$ looks like this:
\[
\sigma_i = \hbar_i \rho^-_i \frac{1- e^{-\rho^+_i}}{h_i \left( e^{-\rho^+_i}-1\right) + h_{i+1}\left(e^{\rho^-_i}-1\right)},
\]
where
\[
\rho^-_i = \frac{b(x_i)h_i}{\eps} \ \ \text{and} \ \ \rho^+_i = \frac{b(x_i)h_{i+1}}{\eps}.
\]
Other possibilities of forming $\sigma_i$ are to replace $D^+v(x_i)$ in \eqref{vexact} with the central approximations of $v'(x_i)$,
\[
D' v(x_i) = \frac{v(x_{i+1}) - v (x_{i-1})}{2\hbar_i}
\]
or
\be
\label{D'cent}
\bar D' v(x_i) = \frac{1}{2\hbar_i}\left[ -\frac{h_{i+1}}{h_i}v(x_{i-1}) + \left(\frac{h_{i+1}}{h_i} - \frac{h_i}{h_{i+1}}\right) v(x_i)  + \frac{h_i}{h_{i+1}} v(x_{i+1})\right].
\ee
Any of the three constructions of $\sigma_i$ described above produces a scheme that is identical to the ASI scheme on the uniform mesh. On a 
nonuniform mesh, however, these schemes are more complicated and harder to analyze than the scheme with $\sigma=\sigma_{\text{ASI}}$. This 
is why they are not considered here anymore.

We denote the truncation error of the scheme $L^N$ at the point $x_i$ by $\tau_i[g]$,
\be
\label{deftrunc}
\tau_i[g] = -\eps\sigma(\rho_i)D''g(x_i) - b(x_i) D^+ g(x_i) + \eps g''(x_i) + b(x_i) g'(x_i),
\ee
where the function $g$ is sufficiently smooth and  $i=1,2,\ldots, N-1$. The truncation error \eqref{deftrunc} can be estimated in different 
ways depending on how far the terms $g(x_{i\pm1})$ within $D''g(x_i)$ and 
$D^+g(x_i)$ are expanded about $x_i$. First, using \eqref{sigma1+}, we easily derive the following estimate from \eqref{deftrunc}:
\be
\label{trunc0}
|\tau_i[g]| \le C\left( \eps \| g''\|_{[x_{i-1}, x_{i+1}]} + \| g' \|_{[x_i, x_{i+1}]}\right).
\ee
Further, it is shown in \cite{VNfc} that, using different notation,
\be
\label{trunc}
|\tau_i[g]| \le  E^0_i[g] + C\left( E^1_i[g] + E^2_i[g]\right),
\ee
where
\[
E^0_i[g] = \eps|\sigma(\rho_i)+\rho_i-1||g''(x_i)|,
\]
\[
E^1_i[g] = h_{i+1}^2 \| g'''\|_{[x_{i}, x_{i+1}]},
\]
and, taking $\sigma(\rho_i)\le \tilde C$ from \eqref{sigma1+} into account,
\[
E^2_i[g] = \eps\left[ (h_{i+1}-h_i)|g'''(x_i)| + h_{1+1}^2 \| g^{(4)}\|_{[x_{i-1}, x_{i+1}]}\right].
\]
Another, first-order, estimate of the truncation error is also derived in \cite{VNfc},
\be
\label{trunc1}
|\tau_i[g]|  \le  Ch_{i+1}\left( \eps \|g'''\|_{[x_{i-1},x_{i+1}]} + \|g''\|_{[x_i, x_{i+1}]} \right).
\ee

\section{The Numerical Method}
\label{sec:nm}

The discretization of the problem \eqref{1DCD} by the general Samarskii-type scheme is
\begin{eqnarray}
\label{discr_u}
 & U^N_0   =   0, \nonumber \\
 & L^N U^N_i  =  f(x_i), \ \ i=1,2,\ldots, N-1, \\
 & U^N_N  =  0. \nonumber
\end{eqnarray}
This system has a unique solution $U^N$ because $L^N$ satisfies the discrete maximum principle. The discretization \eqref{discr_u} on the mesh $S_N$ is analyzed in 
\cite{VNfc} under appropriate conditions ($S_{1/\eps}$ is not considered there). It is proved that $U^N$ is a first-order $\eps$-uniform approximation of $u^N$, 
$\left\| u^N - U^N\right\|\le CN^{-1}$. The same result was previously proved in \cite{Linss02} for 
a class of schemes, including the Samarskii-type ones, derived using the finite-volume framework. 
This estimate can be considerably improved if we discretize the problem \eqref{wprob}, 
\begin{eqnarray}
\label{discr_w}
 & W^N_0   =   -u_0(0), \nonumber \\
 & L^N W^N_i  =  \eps u''_0(x_i), \ \ i=1,2,\ldots, N-1, \\
 & W^N_N  =  0. \nonumber
\end{eqnarray}
The discrete problem \eqref{discr_w} is analyzed in \cite{VNfc} on the mesh $S_N$, and the following result is obtained.

\begin{thm}[{\cite[Remark 2]{VNfc}}]
\label{thm:main_old}
Let $b,c,f\in C^4(I)$ and \eqref{bc_cond} hold, and suppose the solution $u_0$ of the reduced problem \eqref{red} is 
known. Further, let $w$ be the solution of the problem \eqref{wprob} and 
let $W^N$ be the solution of the discrete problem \eqref{discr_w} on the mesh $S_N$ with the parameter $a\ge 6$.
Then, under the assumptions \eqref{sigma1+} and \eqref{sigma2}, it holds that
\[
\|w^N - W^N \| \le C \max\left\{ N^{-2}(\ln N)^2,\ \eps^2 N^{-1} \right\}. 
\]
\end{thm}

Since the term $N^{-2}(\ln N)^2$ usually dominates the above error estimate, the Samarskii-type schemes generally achieve $\eps$-uniform accuracy of almost second 
order. This is confirmed by the numerical results in \cite{VNfc}. However, as mentioned in the introduction, the ASI scheme produces results that are more accurate than 
what Theorem \ref{thm:main_old} indicates, and this is what we want to analyze in the present paper. We consider both $S_N$ and $S_{1/\eps}$ meshes. 
We still discretize \eqref{wprob}, but use the decomposition $w=y+v$ to estimate the error. We initially assume that we know the reduced solution $u_0$.  
The case when $u_0$ is not known and needs to be calculated numerically can be analyzed like in \cite[Subsection 4.1]{VNfc}. This is briefly discussed in Subsection 
\ref{sec:u0_numer} later in the text.

Thus, throughout the rest of the paper, we analyze the ASI discretization of the problem \eqref{wprob}:
\begin{eqnarray}
\label{discr_ASI}
 & W^N_0   =   -u_0(0), \nonumber \\
 & L^N_{\text{ASI}} W^N_i  =  \eps u''_0(x_i), \ \ i=1,2,\ldots, N-1, \\
 & W^N_N  =  0, \nonumber
\end{eqnarray}
where $L^N_{\text{ASI}}$ stands for the scheme $L^N$ with $\sigma = \sigma_{\text{ASI}}$. From now on, the label `ASI' is omitted so that 
$L^N=L^N_{\text{ASI}}$ and $\sigma = \sigma_{\text{ASI}}$.

From \eqref{deftrunc}, we immediately have
\be
\label{deftrunc1}
\tau_i[w] = L^N w(x_i) - f(x_i) = L^N\left[w(x_i) - W^N_i\right],
\ee
which is used to get estimates of $\|w^N - W^N \|$. This requires special techniques instead of the simple $\eps$-uniform stability 
of the scheme $L^N$ because, when $N\to\infty$, the truncation error does not converge to 0 uniformly in $\eps$, \cite{VN14}. We consider two
such techniques here, the barrier-function technique (which we also used to prove Theorem \ref{thm:main_old}) and the preconditioning 
technique.

\section{The Two Proof Techniques}
\label{sec:bar-pre}

We first describe the barrier-function method of proof. For this technique in general, see \cite{KT78, SR97} for instance. 
Regarding some specific details, we follow 
\cite{NV18, VNfc}. The results are somewhat different on $S_N$ and $S_{1/\eps}$, but in both cases, we use the truncation-error 
estimate in the form of
\be
\label{trunc_bar}
|\tau_i[w]| \le \tau^*_i := C^*\left\{ \begin{array}{ll} 
\eta_1 \left(1 + \eps^{-1} e^{-\beta x_i/\eps}\right), & 1 \le i \le J-1, \\
\eta_2, & J\le i \le N-1,
\end{array}\right.
\ee
where $0\le\eta_j\le C$ as $\eps\to 0$ (it is even possible that $\eta_j$ depends on $\eps$ and tends to 0 when $\eps$ does), and $\eta_j\to 0$ as 
$N\to\infty$, $j=1,2$.

\begin{thm}
\label{thm:bar}
Let $w$ be the solution of the problem \eqref{wprob} satisfying \eqref{bc_cond}, and 
let $W^N$ be the solution of the discrete problem \eqref{discr_w} on the mesh $S_\lambda$.
If the truncation error satisfies \eqref{trunc_bar}, then it holds that
\[
\|w^N - W^N \| \le C\eta,
\]
where $\eta=\max\{\kappa \eta_1, \eta_2\}$ and
\[
\kappa = \left\{ \begin{array}{ll}
1 & \text{if } \ \lambda = N, \\
1+ |\ln\eps|N^{-1} & \text{if } \ \lambda=\eps^{-1}.
\end{array}\right.
\]
\end{thm}
\begin{proof}
Let
\[
B^N_0 = 1, \ \ B^N_i = \prod_{j=1}^{i}\left(1+\frac{\beta h_j}{3\eps}\right)^{-1}, \ \ i=1,2,\ldots, N.
\]
We have that
\[ 
e^{-\beta x_i/\eps} \le e^{-\beta x_i/(3\eps)} = \prod_{j=1}^{i}\left(e^{-\beta h_j/(3\eps)}\right) \le B^N_i. 
\]

We also define the barrier function
\[
\gamma^N_i= \eta\left(\gamma^{1,N}_i + \gamma^{2,N}_i\right), \ \ i=0,1,\ldots, N,
\]
where
\[
\gamma^{1,N}_i:=C_1(1-x_i) \ \text{ and } \ \gamma^{2,N}_i:=C_2 B^N_i 
\]
with appropriately chosen constants $C_1$ and $C_2$. We immediately have
\[
\gamma^N_i \ge 0 = \left| w(x_i) - W^N_i\right|, \ \ i =0, N.
\]
It also holds that $\|\gamma^N\| \le C\eta$. Therefore, keeping \eqref{deftrunc1} and the maximum principle in mind, we only need to 
prove the inequality 
\be
\label{comp}
L^N \gamma^N_i \ge \tau^*_i, \ \ 1\le i \le N-1.
\ee

First,
\be
\label{complin}
L^N \gamma^{1,N}_i \ge C_1\beta \ge C^*, \ \ 1\le i \le N-1,
\ee
if $C_1$ is chosen so that $C_1\ge C^*/\beta$. Next, for $2\le i\le N-1$, we have
\[
L^N B^N_i \ge \left[ - \frac{\sigma(\rho_i)}{\hbar_i} \cdot\frac{(\beta/3)^2 h_{i+1}}{\eps+(\beta/3) h_{i+1}} 
+ b(x_i)\frac{\beta}{3\eps+\beta h_{i+1}}\right] B^N_i.
\]
From there, because of \eqref{sigma1+} with $\tilde C=1$, we get
\begin{eqnarray}
\label{compB}
L^N B^N_i & \ge & \frac{\beta}{3\eps+\beta h_{i+1}}\left[ b(x_i) - \beta\frac{h_{i+1}}{3\hbar_i}\right] B^N_i \nonumber \\
& \ge & \frac{\beta^2}{3(3\eps+\beta h_{i+1})} B^N_i \ \ (\text{using } \hbar_i \ge \frac{h_{i+1}}{2}).
\end{eqnarray}
The case $i=1$ requires the use of $B^N_0=1$, but the inequality \eqref{compB} is still true, see \cite[Remark 1]{VNfc}.
Because of \eqref{compB}, $L^N B^N_i \ge 0$, which combined with \eqref{complin} gives 
\be
\label{comp2}
L^N \gamma^N_i \ge C^* \eta \ge C^*\eta_2 = \tau^*_i, \ \ J \le i \le N-1.
\ee

When $1\le i \le J-1$, \eqref{compB} implies that
\[
L^N \gamma^{2,N}_i \ge  C_2\frac{\beta^2}{3(3\eps+\beta h)} B^N_i \ge  C_2\frac{C_*}{\eps\kappa}B^N_i
\]
with some constant $C_*$. We now choose $C_2\ge C^*/C_{*}$ and combine the above with \eqref{complin} to get
\be
\label{comp1}
L^N \gamma^N_i \ge C^*\kappa \eta_1 \left( 1 + \frac{1}{\eps\kappa} B^N_i\right) \ge \tau^*_i, \ \ 1\le  i \le J-1.
\ee
Then, \eqref{comp} follows from \eqref{comp2} and \eqref{comp1}.
\end{proof}

We now turn to the preconditioning technique of \cite{VN14}. Only the upwind scheme on the mesh $S_N$ is 
considered in \cite{VN14}; $S_{1/\eps}$ is not discussed. In the preconditioning approach, the discrete equations at $x_i$, $1\le i\le J$, 
are multiplied by $\hbar_i/H$. This renders the discrete system well-conditioned (otherwise, the condition number grows unboundedly when 
$\eps\to 0$), hence the name of the technique.
At the same time, the truncation error becomes almost-first-order accurate uniformly in $\eps$, a property that the 
upwind scheme generally does not have on Shishkin-type meshes. 
In \cite{VN21}, a slight modification of the original approach is used on the mesh $S_{1/\eps}$. We apply the same 
modification to both meshes here. Specifically, we scale the discrete problem \eqref{discr_ASI} as follows:
\begin{eqnarray}
\label{discr_pre}
 & W^N_0  =  -u_0(0), \nonumber \\
& \tilde L^N W^N_i := m_i L^N W^N_i  =  m_i\eps u''_0(x_i), \ \ i=1,2,\ldots, N-1, \\
& W^N_N  =  0, \nonumber
\end{eqnarray}
where 
\[
m_i = \left\{ \begin{array}{ll}
\dss\frac{h}{H}, & 1\le i \le J-1, \\
& \\
1, & J\le i \le N-1.
\end{array}\right.
\]
The truncation error of the operator $\tilde L^N$ is
\[
\tilde \tau_i[w] = m_i \tau_i[w], \ \  1\le i \le N-1. 
\]
In the preconditioning approach, the factor $e^{-\beta x_i/\eps}$ in the estimate of $|\tau_i[w]|$ is no longer needed. 
We assume that we have the following:
\be
\label{trunc_pre}
|\tau_i[w]| \le C' \left\{ \begin{array}{ll} 
\tilde\eta_1\eps^{-1}, & 1 \le i \le J-1, \\
\tilde\eta_2, & J\le i \le N-1.
\end{array}\right.
\ee

\begin{thm}
\label{thm:pre}
Let $w$ be the solution of the problem \eqref{wprob} satisfying \eqref{bc_cond}, and 
let $W^N$ be the solution of the discrete problem \eqref{discr_w} on the mesh $S_\lambda$.
If the truncation error satisfies \eqref{trunc_pre}, then it holds that
\[
\|w^N - W^N \| \le C\tilde\eta,
\]
where $\tilde\eta=\max\{\tilde\eta_1 \ln\lambda, \tilde\eta_2\}$.
\end{thm}
\begin{proof}
We apply the technique from \cite[proof of Theorem 4]{VN21} to the discrete problem \eqref{discr_pre} 
(which is equivalent to \eqref{discr_w}). This technique is used in the matrix form in \cite{VN21} for
the upwind scheme on the mesh $S_{1/\eps}$. We apply it here to both meshes using the discrete operator form, which makes the steps closer 
to those in the barrier-function approach.

Since $h/H\le C\eps\ln\lambda$, we have from \eqref{trunc_pre} that
\[
|\tilde\tau_i[w]| \le \hat C\tilde\eta, \ \ 1\le i\le N-1. 
\]
We define the barrier function
\[
\tilde\gamma^N_i = \frac{\hat C}{\beta}\tilde\eta \hat\gamma^N_i,
\]
with
\[
\hat\gamma^N_i :=  \hat C_1 -Hi + \hat C_2\min\left\{ (1+ R)^{J-i},\ 1\right\}, \ \ R=\frac{\beta H}{\eps}.
\]
The constants $\hat C_1$ and $\hat C_2$ are to be determined so that $\hat C_1 -Hi \ge 0$ (thus $\hat\gamma_i^N \ge 0$) and
\be
\label{hatgam}
\tilde L^N \hat\gamma^N_i \ge \beta
\ee
for $1\le i \le N-1$. Then, it immediately follows that
\[
\tilde \gamma^N_i \ge |\tilde\tau_i[w]|,
\]
which implies the assertion because $\tilde\gamma^N_i \le C\tilde\eta$. 

The rest of the proof is about using 
\[
\tilde L^N \hat \gamma^N_i \ge m_i \left[-\eps\sigma(\rho_i) D'' \hat\gamma^N_i - b(x_i)D^+ \hat\gamma^N_i\right].
\]
to show that \eqref{hatgam} is true. 

Let us consider $1\le i\le J-1$. Then, $\hat\gamma_j^N = \hat C_1 -Hj+ \hat C_2$ for $j=i-1, i, i+1$, and it follows that
\[
\tilde L^N \hat \gamma^N_i  \ge -\frac{\eps\sigma(\rho_i)}{hH}\cdot H + \frac{\eps\sigma(\rho_i)}{hH}\cdot H + b(x_i) \ge \beta, 
\]
which proves \eqref{hatgam} in this case.

If $J\le i \le N-1$, then $\tilde L^N \hat\gamma^N_i = L^N \hat\gamma^N_i$.  For $i=J$, we get
\begin{eqnarray*}
\tilde L^N \hat \gamma^N_i  & \ge & -\frac{\eps\sigma(\rho_i)H}{h\hbar_i} + \frac{\eps\sigma(\rho_i)}{\hbar_i} + b(x_i) 
+ \hat C_2\left[\frac{\eps\sigma(\rho_i)}{H\hbar_i}+\frac{b(x_i)}{H}\right]\frac{R}{1+R}\\
&\ge & -\frac{\eps\sigma(\rho_i)H}{h\hbar_i} + \beta + \hat C_2\frac{\eps\sigma(\rho_i)+\beta\hbar_i}{\hbar_i}\frac{\beta}{\eps + \beta H}\\
&= & \frac{1}{\hbar_i}\left[\hat C_2\beta \frac{\eps\sigma(\rho_i)+\beta\hbar_i}{\eps+\beta H}- \frac{\eps\sigma(\rho_i)H}{h}\right] + \beta.
\end{eqnarray*}
Next, we have
\[
\eps\sigma(\rho_i)+\beta\hbar_i \ge \frac12 [\eps\sigma(\rho_i)+\beta H] \ge 
\frac12 [\eps\sigma(\rho_i)+\beta\sigma(\rho_i) H],
\]
where the last inequality is true because of $1\ge \sigma(\rho_i)$, see Subsection \ref{sec:Sam}. This implies
\[
\tilde L^N \hat \gamma^N_i \ge \frac{\sigma(\rho_i)}{\hbar_i}\left[\frac{\hat C_2\beta}{2} - \frac{\eps H}{h}\right] + \beta.
\]
Since
\[
\frac{\eps H}{h} \le \frac{C}{\ln\lambda},
\]
$\hat C_2$ can be selected so that
\[
\frac{\hat C_2\beta}{2} - \frac{\eps H}{h} \ge 0.
\]
Thus, \eqref{hatgam} holds in this case as well. 

Let finally $J+1\le i \le N-1$. Then, \eqref{hatgam} follows from
\[
\tilde L^N \hat \gamma^N_i \ge b(x_i) +\hat C_2(1+R)^{J-i-1} P_i \ge \beta
\]
because
\begin{eqnarray*}
P_i & := & -\frac{\eps\sigma(\rho_i)}{H^2}[(1+R)^2 - (1+R)] - \left[ \frac{\eps\sigma(\rho_i)}{H^2} + \frac{b(x_i)}{H}\right] [1 -(1+R)] \\
& = & R\left[ -\frac{\eps\sigma(\rho_i)}{H^2}(1+R) + \frac{\eps\sigma(\rho_i)}{H^2} + \frac{b(x_i)}{H}\right] \\
& = & R\left[ -\frac{\beta \sigma(\rho_i)}{H} +  \frac{b(x_i)}{H}\right] \\
& \ge & \frac{R\beta}{H}[1- \sigma(\rho_i)] \\
& \ge & 0.
\end{eqnarray*}
\end{proof}

\begin{remark}
\label{rem:bar-pre}
The estimate \eqref{trunc_pre} follows from \eqref{trunc_bar} with $\tilde\eta_j=\eta_j$, $j=1,2$. Then, the result of Theorem 
\ref{thm:pre} on $S_N$ is worse than the result of Theorem \ref{thm:bar} because there is an extra $(\ln N)$-factor 
within the error estimate of Theorem \ref{thm:pre}, and on $S_{1/\eps}$ because $1+|\ln\eps|N^{-1}\le C|\ln\eps|$. 
However, it may be possible to derive \eqref{trunc_pre} in a different way, not from \eqref{trunc_bar}, and then $\tilde\eta_j$ and 
$\eta_j$ are not necessarily equal. We shall see later that this is possible on $S_{1/\eps}$.
\end{remark}

\section{The Truncation-Error Estimates}
\label{sec:trunc}

We first analyze the truncation error for the $v$-component of the solution decomposition. The following estimate from 
\cite[p. 1035]{KT78} is crucial in the analysis:  
\be
\label{trunc_vKT}
|\tau_i[v]| \le C \frac{\chi^2}{\eps(\chi+\eps)}e^{-\beta x_i/\eps},
\ee
where $x_i=\chi i$ and $\chi=1/N$ is the step of an arbitrary uniform mesh (the only mesh considered in \cite{KT78}). 

\begin{remark}
\label{rem:KTlemma}
The estimate \eqref{trunc_vKT} is derived as a step in the proof of Lemma 4.3 in \cite{KT78}. The proof of \eqref{trunc_vKT} is rather technical, and its 
generalization to the nonuniform seems very hard. This is why Bakhvalov-type meshes are not considered here. The piecewise uniform 
Shishkin mesh, on the other hand, is suitable for our analysis; we only need to find a different way of estimating $|\tau_J[v]|$.
\end{remark}

Before proceeding, we define
\[
a_* = \frac{\beta}{b(0)-\beta}= \frac{1}{q-1}, \ \ q=\frac{b(0)}{\beta}>1.
\]

\begin{lemma} 
\label{lem:trunc_vSN}
Let $b,c,f\in C^4(I)$ and \eqref{bc_cond} hold.
Then, the truncation error of the ASI scheme on the mesh $S_N$ with $a\ge a_*$ satisfies
\[
|\tau_i[v]| \le C \left\{\begin{array}{ll}
N^{-2}(\ln N)^2 e^{-\beta x_i/\eps}, & 1\le i \le J-1,  \\
N^{-a}, & J\le i \le N-1.
\end{array}\right.
\]
\end{lemma}
\begin{proof}
We consider $S_\lambda$ whenever the steps of the analysis apply equally to $S_N$ and $S_{1/\eps}$.

Let us first discuss $1 \le i \le J-1$. We use \eqref{trunc_vKT}, but in this case, $\chi$ is the step $h$ of the fine part of the mesh. Therefore,
\begin{eqnarray}
\label{trunc_v}
|\tau_i[v]| & \le & C \frac{\left( \eps N^{-1} \ln\lambda\right)^2}{\eps \left(\eps N^{-1} \ln\lambda + \eps\right)}e^{-\beta x_i/\eps}
\nonumber \\ 
 & = & C \frac{\left( N^{-1} \ln\lambda\right)^2}{N^{-1} \ln\lambda + 1}e^{-\beta x_i/\eps} \nonumber \\
& \le & C N^{-2} (\ln\lambda)^2 e^{-\beta x_i/\eps}.
\end{eqnarray}

When $J+1 \le i \le N-1$, an important fact we need is
\be
\label{eleft_coarse}
e^{-\beta x_{i}/\eps} \le e^{-\beta(\xi+H)/\eps} = \lambda^{-a}e^{-\beta H/\eps},
\ee
cf.\ \cite[proof of Lema 2]{VNfc}. We use \eqref{trunc_vKT} again, but now,
\[
|\tau_i[v]| \le C \frac{H^2}{\eps(H+\eps)}e^{-\beta x_i/\eps}.
\]
Then, from \eqref{eleft_coarse}, we have
\be
\label{trunc_v2}
|\tau_i[v]| \le C \frac{H^2}{\eps^2} \lambda^{-a} e^{-\beta H/\eps},
\ee
which implies
\be
\label{trunc_v3}
|\tau_i[v]| \le C \lambda^{-a}
\ee
since $H^2\eps^{-2}e^{-\beta H/\eps}\le C$.

We finally consider $x_i=x_J=\xi$. We use \eqref{deftrunc} with $g=v$, noting that
\be
\label{trunc_v4}
|\eps v''(\xi) + b(\xi)v'(\xi)| \le C \frac{\xi}{\eps} e^{-b(0)\xi/\eps} = C \frac{\xi}{\eps} e^{[\beta-b(0)]\xi/\eps} e^{-\beta \xi/\eps} 
\le C \lambda^{-a}. 
\ee
We begin estimating the remaining terms in \eqref{deftrunc} as follows:
\begin{eqnarray}
\label{trunc_v5}
|\eps\sigma(\rho_i)D''v(x_J) + b(x_J) D^+ v(x_J)| & \le & C\left( \frac{\eps}{\hbar} \| v'\|_{[x_{J-1},x_{J+1}]} 
+ \frac{1}{H} \| v\|_{[x_J,x_{J+1}]}\right) \nonumber \\ 
& \le & C H^{-1} e^{-b(0)x_{J-1}/\eps} \nonumber \\
& \le & C N e^{-b(0)\xi/\eps} e^{b(0)h/\eps} \nonumber \\
& \le & C N \lambda^{-aq+aq/J},
\end{eqnarray}
and finish on $S_N$ by setting $\lambda=N$ and using $a\ge a_*$:
\[
N \lambda^{-aq+aq/J}  = N^{1 - aq} e^{(aq\ln N)/(QN)}  \le  C N^{-a}.
\]
\end{proof}

\begin{lemma} 
\label{lem:trunc_vSE}
Let $b,c,f\in C^4(I)$ and \eqref{bc_cond} hold.
Then, the truncation error of the ASI scheme on the mesh $S_{1/\eps}$ with $a\ge 2$ satisfies
\[
|\tau_i[v]| \le C\left\{\begin{array}{ll}
(\ln \eps)^2 N^{-2} e^{-\beta x_i/\eps}, & 1\le i \le J-1,  \\
\mu N, & i=J, \\
\mu, & J+1\le i \le N-1,
\end{array}\right.
\]
where $\mu = \eps^{a-2}\min\left\{\eps^2, N^{-2}\right\}$ and $N$ is sufficiently large independently of $\eps$.
\end{lemma}
\begin{proof}
We use many general estimates on $S_\lambda$ from the proof of Lemma \ref{lem:trunc_vSN} and set $\lambda=\eps^{-1}$.
The estimate for $1\le i \le J-1$ follows from \eqref{trunc_v}. 
If $J+1 \le i \le N-1$, we use \eqref{trunc_v2}:
\[
|\tau_i[v]| \le C \frac{H^2}{\eps^2} \eps^{a} e^{-\beta H/\eps} \le C \eps^{a-2} N^{-2}.
\]
We combine this with \eqref{trunc_v3} and get the desired estimate.

In the remaining case, $i=J$, we start from \eqref{trunc_v5} as follows:
\[
|\eps\sigma(\rho_J)D''v(x_J) + b(x_J) D^+ v(x_J)| \le C N \eps^{aq-aq/J} \le  C N \eps^a
\]
where the last step is true if $J$ (that is, $N$) is sufficiently large independently of $\eps$.
This combined with \eqref{trunc_v4} gives
\be
\label{trunc_v6}
|\tau_J[v]| \le C N\eps^a.
\ee
We finish using \eqref{trunc1}:
\[
|\tau_J[v]| \le C N^{-1} \eps^{-2} e^{-b(0)x_{J-1}/\eps} \le C \eps^{a-2} N^{-1},
\]
where $N$ has to be sufficiently large again. This and \eqref{trunc_v6} complete the proof of the last estimate.
\end{proof}

\begin{remark}
\label{rem:lem_trunc_vSE}
If in the above proof of Lemma \ref{lem:trunc_vSE} in the case when $1\le i\le J-1$, we use \eqref{trunc_vKT} differently, we get
\[
|\tau_i[v]|  \le C \frac{h}{\eps}e^{-\beta x_i/\eps} \le C |\ln\eps|  N^{-1} \cdot e^{-\beta x_i/\eps}.
\]
Thus, if $a\ge 3$, we have that
\be
\label{trunc_vSE_lin}
|\tau_i[v]| \le C\left\{\begin{array}{ll}
|\ln \eps| N^{-1} e^{-\beta x_i/\eps}, & 1\le i \le J-1,  \\
\eps N^{-1}, & J+1\le i \le N-1.
\end{array}\right.
\ee
\end{remark}

Let us now consider the $y$-component of the solution. It would be ideal if we could prove the same kind of estimates
for $|\tau_i[y]|$ as for $|\tau_i[v]|$. However, there are some technical difficulties to do this, particularly on the mesh $S_N$ when $i=J$ 
and even when $i=J+1$. 

\begin{lemma} 
\label{lem:trunc_ySN}
Let $b,c,f\in C^4(I)$ and \eqref{bc_cond} hold.
Then, the truncation error of the ASI scheme on the mesh $S_N$ with $a\ge 1$ satisfies
\[
|\tau_i[y]| \le C \left\{\begin{array}{l}
\eps N^{-2}(\ln N)^2\left(\eps + \eps^{-1}e^{-\beta x_i/\eps}\right), \ \ 1\le i \le J-1, \\
\left( \eps N^{-1} +  N^{-(a+1)/2}\right), \ \ i=J,J+1, \\
\left(\min\{\eps, N^{-1}\}N^{-1} + N^{-a}\right), \ \ J+2\le i \le N-1.
\end{array}\right.
\]
\end{lemma}
\begin{proof}
We consider the general mesh $S_\lambda$ whenever the resulting estimate can be useful later.
The derivative estimates \eqref{dery} are needed throughout the proof.

Let $1\le i \le J-1$. We handle the first component $E^0_i[y]$ of the truncation error in \eqref{trunc} using 
$|\sigma(\rho) + \rho - 1| \le \bar C\rho^2$ from \eqref{sigma2},
\begin{eqnarray*}
E^0_i[y] & \le & C\eps \rho_i^2\left(\eps + \eps^{-1}e^{-\beta x_i/\eps}\right) \\
& \le & C h^2 \eps^{-2} \left(\eps^2 + e^{-\beta x_i/\eps}\right) \\
& \le & C N^{-2}(\ln\lambda)^2 \left(\eps^2 + e^{-\beta x_i/\eps}\right).
\end{eqnarray*}
Next,
\begin{eqnarray*} 
E^1_i[y] & \le & C h^2 \left(\eps + \eps^{-2}e^{-\beta x_i/\eps}\right) \\
& \le & C N^{-2} (\ln\lambda)^2 \left(\eps^3 + e^{-\beta x_i/\eps}\right).
\end{eqnarray*} 
We finally use 
\be
\label{eleft}
e^{-\beta x_{i-1}/\eps} = e^{-\beta x_i/\eps} e^{\beta h/\eps} = \lambda^{a/J} e^{-\beta x_i/\eps}
\ee
to get
\begin{eqnarray*} 
E^2_i[y] & \le & C \eps h^2 \left(\eps + \eps^{-3}e^{-\beta x_{i-1}/\eps}\right) \\
& \le & C N^{-2} (\ln\lambda)^2 \left(\eps^4 + \lambda^{a/J} e^{-\beta x_i/\eps}\right).
\end{eqnarray*}
Therefore, from \eqref{trunc}, we have
\be
\label{trunc_yleft}
|\tau_i[y]| \le C N^{-2} (\ln\lambda)^2 \left(\eps^2 + \lambda^{a/J}e^{-\beta x_i/\eps}\right), \ \ 1\le i \le J-1.
\ee
Setting $\lambda = N$, we complete this part of the proof since $N^{a/J}\le C$.

Let us now consider $J+2\le i\le N-1$. We use the truncation error in the form of \eqref{trunc} in this case as well. 
We have
\begin{eqnarray*}
E^0_i[y] & \le & C\eps \rho_i^s \left(\eps + \eps^{-1}e^{-\beta x_i/\eps}\right) \nonumber \\
& \le & C H^s \eps^{-s} \left(\eps^2 + \lambda^{-a}e^{-\beta H/\eps}\right) \ \ \text{(because of \eqref{eleft_coarse})} \nonumber \\
& \le & C \left(H^s\eps^{2-s} + \bar E^s  \right),
\end{eqnarray*}
where
\[ 
\bar E^s := \lambda^{-a}\left(\frac{H}{\eps}\right)^s e^{-\beta H/\eps}
\]
and $s=1$ or $s=2$. It holds that
\[
H^s\eps^{2-s} \le C\left\{\begin{array}{ll}
\eps N^{-1} & \text{if } \ s=1, \\
 N^{-2} & \text{if } \ s=2, 
\end{array}\right.
\]
which implies
\be
\label{y_coarse0}
E^0_i[y] \le C \min\left \{\eps N^{-1} + \bar E^1,\ N^{-2} + \bar E^2 \right\}.
\ee
Also,
\begin{eqnarray}
\label{y_coarse1}
E^1_i[y] & \le & CH^2 \left( \eps + \eps^{-2} \lambda^{-a} e^{-\beta H/\eps}\right) \nonumber \\
& \le & C\left( \eps N^{-2} + \bar E^2 \right)
\end{eqnarray}
and
\begin{eqnarray}
\label{y_coarse2}
E^2_i[y] & \le & C \eps H^2\left(\eps + \eps^{-3}e^{-\beta x_{i-1}/\eps}\right) \nonumber \\ 
& \le & C \left( \eps^2 N^{-2} + \bar E^2 \right). 
\end{eqnarray}
From this point in the proof, we only consider $S_N$. When we set $\lambda=N$, we get $\bar E^s \le N^{-a}$, 
$s=1,2$, and the desired estimate follows from \eqref{y_coarse0}, \eqref{y_coarse1}, and \eqref{y_coarse2}. 

Finally, let $i=J,J+1$. Instead of \eqref{eleft_coarse}, we now only have
\[
e^{-\beta x_{i-1}/\eps} \le e^{-\beta(\xi-h)/\eps} \le CN^{-a}.
\]
Suppose
\be
\label{epsN}
\eps \ge N^{-(a+1)/2},
\ee
in which case we use \eqref{trunc1} to estimate the truncation error. We have
\[
|\tau_i[y]| \le C N^{-1}\left[ \eps\left(\eps + \eps^{-2}N^{-a}\right) + \eps + \eps^{-1}N^{-a}\right].
\]
This and \eqref{epsN} imply the desired estimate:
\[
|\tau_i[y]| \le C \left( \eps N^{-1} + \eps^{-1}N^{-a-1}\right) \le C \left( \eps N^{-1} + N^{-(a+1)/2}\right).
\]

It remains to consider $i=J,J+1$ when $\eps\le N^{-(a+1)/2}$. We use \eqref{trunc0} to get
\[
|\tau_i[y]| \le C\left(\eps + N^{-a}\right) \le C \left( N^{-(a+1)/2} + N^{-a}\right) \le C N^{-(a+1)/2} \ \ \text{(since $a\ge 1$)}.
\]
\end{proof}

\begin{remark}
\label{rem:trunc_ySN}
It is possible to use \eqref{trunc} instead of \eqref{trunc1} in the above proof of the estimate for $i=J,J+1$, and the result is 
somewhat different from that in Lemma \ref{lem:trunc_ySN}. If $a\ge 2$, then
\be
\label{trunc_ySNrem}
|\tau_i[y]| \le C\left( \min\{\eps, N^{-1}\}N^{-1} + \eps^2 N^{-1 + J-i}  + N^{-(a+2)/3}\right)
\ee
for $i=J,J+1$. When we compare this to the estimate in Lemma \ref{lem:trunc_ySN}, we see that $\min\{\eps, N^{-1}\}N^{-1} \le \eps N^{-1}$ 
and $\eps^2 N^{-1 + J-i} \le \eps N^{-1}$, but $N^{-(a+2)/3} \ge N^{-(a+1)/2}$. Moreover, in practical situations, we usually have
$\eps\le N^{-1}$, thus we can expect the estimate in \eqref{trunc_ySNrem} to behave like $C\left( \eps N^{-1} + N^{-(a+2)/3}\right)$,
and the estimate in Lemma \ref{lem:trunc_ySN} behaves better.
\end{remark}

Let us now consider $\tau_i[y]$ on the mesh $S_{1/\eps}$. We only provide a full analysis leading to an estimate in the form \eqref{trunc_pre}
suitable for the preconditioning approach. 

\begin{remark}
\label{rem:barSE}
The barrier-function technique does not give satisfactory results on $S_{1/\eps}$.
The factor $\kappa = 1+|\ln\eps|N^{-1}$ in the error-estimate of Theorem \ref{thm:bar} is one reason for this. 
Another problem is that \eqref{eleft} becomes
\[
e^{-\beta x_{i-1}/\eps} = \eps^{-a/J} e^{-\beta x_i/\eps}
\]
and gives
\[
|\tau_i[y]| \le C \eps^{1-a/J}(\ln\eps)^2 N^{-2}\left(\eps^{1+a/J} + \eps^{-1}e^{-\beta x_i/\eps}\right), \ \ 1\le i \le J-1,
\]
which is worse than what we have in Lemma \ref{lem:trunc_ySN}.
\end{remark}

\begin{lemma} 
\label{lem:trunc_ySE}
Let $b,c,f\in C^4(I)$ and \eqref{bc_cond} hold.
Then, the truncation error of the ASI scheme on the mesh $S_{1/\eps}$ with $a> 3$ satisfies
\[
|\tau_i[y]| \le C \left\{\begin{array}{l}
(\ln\eps)^2 N^{-2}, \ \ 1\le i \le J-1, \\
\left( \min\{\eps, N^{-1}\}N^{-1} + \eps^2 N^{-1} \right), \ \ i=J, \\
\min\{\eps, N^{-1}\}N^{-1}, \ \ J+1\le i \le N-1,
\end{array}\right.
\]
provided $N$ is sufficiently large independently of $\eps$.
\end{lemma}
\begin{proof}
When $1\le i \le J-1$, we use \eqref{trunc} like in the proof of Lemma \ref{lem:trunc_ySN} to get
\[
E^0_i[y]  \le  C\eps \rho_i^2\left(\eps + \eps^{-1}\right) \le C h^2 \eps^{-2} \le  C (\ln\eps)^2 N^{-2},
\]
\[
E^1_i[y]  \le  C h^2 \left(\eps + \eps^{-2}\right) \le C (\ln\eps)^2 N^{-2},
\]
and
\[
E^2_i[y]  \le  C \eps h^2 \left(\eps + \eps^{-3}\right) \le C (\ln\eps)^2 N^{-2}.
\]

For $J+2\le i \le N-1$, we set $\lambda=\eps^{-1}$ in \eqref{y_coarse0}, \eqref{y_coarse1}, and \eqref{y_coarse2}. Then the desired estimate 
follows from
\[
\bar E^s = \eps^{a}\left(\frac{H}{\eps}\right)^s e^{-\beta H/\eps} \le C \eps^{a-s} N^{-s}, \ \ s=1,2.
\]
Since the above $e^{-\beta H/\eps}$-factor is not needed on $S_{1/\eps}$, we can also treat $i=J+1$ like above.
This leaves us with $i=J$. 

When $x_i=x_J=\xi$, we proceed with \eqref{trunc} and use
\[
e^{-\beta \xi/\eps} = \eps^a \ \text{ and } \ e^{-\beta x_{J-1}/\eps}=e^{-\beta \xi/\eps}e^{\beta h/\eps}=\eps^{a-a/J}.
\]
It holds that
\begin{eqnarray*}
E^0_J[y] & \le & C\eps \rho_J^s \left(\eps + \eps^{-1}\eps^{a}\right) \\
& \le & C N^{-s}\left(\eps^{2-s} + \eps^{a-s}\right) \\ 
& \le & C\eps^{2-s}N^{-s}, \ \ s=1,2.
\end{eqnarray*}
Therefore,
\be
\label{E0_yJ}
E^0_J[y] \le C \min\{\eps, N^{-1}\}N^{-1}.
\ee
We also have
\be
\label{E1_yJ}
E^1_J[y] \le C N^{-2} \left(\eps + \eps^{-2}\eps^{a}\right) \le  C \eps N^{-2}
\ee
and
\[
E^2_J[y] \le C\eps\left[ N^{-1}\left( \eps + \eps^{-2} \eps^{a}\right) + N^{-2}\left(\eps + \eps^{-3}\eps^{a-a/J}\right)\right].
\]
Since $a- \frac{a}{J} - 3\ge 0$ if $N$ is sufficiently large, it follows that
\be
\label{E2_yJ}
E^2_J[y]  \le  C\left( \eps^2 N^{-1} + \eps N^{-2}\right). 
\ee
Then, \eqref{E0_yJ}, \eqref{E1_yJ},  and \eqref{E2_yJ} imply the estimate for $i=J$.
\end{proof}

\begin{remark}
\label{rem:lem_trunc_ySE}
If in the above proof we only use \eqref{trunc1} to estimate $\tau_i[y]$, we need $a>2$ to get
\be
\label{trunc_ySE_lin}
|\tau_i[y]| \le C \left\{\begin{array}{l}
|\ln\eps| N^{-1}, \ \ 1\le i \le J-1, \\
\eps N^{-1}, \ \ J\le i \le N-1,
\end{array}\right.
\ee
This follows from
\[
|\tau_i[y]|  \le  C h \left( \eps + \eps^{-1}\right) \le C |\ln\eps| N^{-1}, \ \ 1\le i\le J-1,
\]
and
\begin{eqnarray*}
|\tau_i[y]|  & \le &  C N^{-1}\left( \eps + \eps^{-1}e^{-\beta x_{i-1}/\eps} \right) \\
& \le & C N^{-1}\left( \eps + \eps^{a-1-a/J} \right) \le C\eps N^{-1}, \ \ J \le i \le N-1.
\end{eqnarray*}
\end{remark}

\section{The Main Theoretical Results}
\label{sec:main}

In this section, we derive estimates of $|\tau_i[w]|$ from $\tau_i[w] = \tau_i[v] + \tau_i[y]$ and Lemmas 
\ref{lem:trunc_vSN}--\ref{lem:trunc_ySE}. Then we use Theorems \ref{thm:bar} and \ref{thm:pre} to estimate the error of our numerical 
method.

\begin{thm}
\label{thm:main_SN}
Let $b,c,f\in C^4(I)$ and \eqref{bc_cond} hold. Further, let $w$ be the solution of the problem \eqref{wprob} and 
let $W^N$ be the solution of the discrete problem \eqref{discr_w} on the mesh $S_N$ with 
$a\ge \min\left\{ 3, \max\{1, a_*\}\right\}$. Then it holds that
\[
\|w^N - W^N \| \le C\nu_1, 
\]
where
\[
\nu_1 = \max\left\{ \eps N^{-2}(\ln N)^2,\ \eps N^{-1} + N^{-(a+1)/2}\right\}.
\]
\end{thm}
\begin{proof}
We show first that
\be
\label{trunc_wSN}
|\tau_i[w]| \le C \left\{\begin{array}{l}
\eps N^{-2}(\ln N)^2\left(\eps + \eps^{-1}e^{-\beta x_i/\eps}\right), \ \ 1\le i \le J-1, \\
\left( \eps N^{-1} + N^{-(a+1)/2}\right), \ \ i= J,J+1, \\
\left(\min\{\eps, N^{-1}\}N^{-1} + N^{-a}\right), \ \ J+2\le i \le N-1.
\end{array}\right.
\ee
If $a\ge \max\{1, a_*\}$, we get the desired estimate by combining Lemmas \ref{lem:trunc_vSN} and \ref{lem:trunc_ySN}. 
Note that when $i=J,J+1$, we use $N^{-a}\le N^{-(a+1)/2}$, which is true because of $a\ge 1$. 
If, on the other hand, $a\ge 3$, then $a\ge a_*$ is not necessarily satisfied, and the technique used to prove the 
estimate $\tau_J[v]\le C N^{-a}$ in Lemma \ref{lem:trunc_vSN} no longer works. However, we have that
\[
\tau_J[v]\le C N^{1-a}
\]
because of
\[
|\eps\sigma(\rho_J)D''v(x_J) + b(x_J) D^+ v(x_J)| \le C H^{-1} e^{-b(0)\xi/\eps} \le C N e^{-\beta\xi/\eps} = CN^{1-a}.
\]
Then, $a\ge 3$ implies that $N^{1-a}\le N^{-(a+1)/2}$.

Therefore, we have a truncation-error estimate in the form of \eqref{trunc_bar} with
\[
\eta_1 = \eps N^{-2}(\ln N)^2 \ \text{ and } \ \eta_2 = \eps N^{-1} + N^{-(a+1)/2}.
\]
The assertion then follows from Theorem \ref{thm:bar}.
\end{proof}

\begin{remark}
\label{rem:max-sum}
Since it holds that
\[
\max\{A, B\} \le A + B \le 2\max\{A, B\}
\]
for any two non-negative quantities $A$ and $B$, the error estimate $\nu_1$ in Theorem \ref{thm:main_SN} can also be expressed as
\[
\nu_1 = \max\left\{ \eps N^{-2}(\ln N)^2,\ \eps N^{-1},\  N^{-(a+1)/2}\right\}
\]
or
\[
\nu_1 = \eps N^{-2}(\ln N)^2 + \eps N^{-1} + N^{-(a+1)/2}.
\]
In general, we keep the form of the error estimate that results directly from the corresponding truncation-error estimate in Section \ref{sec:trunc}.
\end{remark}

\begin{remark}
\label{rem:main_SN}
Since $\eps\le N^{-1}$ usually holds in practice, the error estimates in Theorems \ref{thm:main_old} and 
\ref{thm:main_SN} typically behave like $C N^{-2}(\ln N)^2$ and 
$C\left( \eps N^{-1}+ N^{-(a+1)/2}\right)$ respectively. Thus, if $a\ge 3$, the 
result of Theorem \ref{thm:main_SN} is better than that of Theorem \ref{thm:main_old}.
Moreover, the condition on the mesh parameter $a$
is less restrictive in Theorem \ref{thm:main_SN} than in Theorem \ref{thm:main_old}.
If $\eps$ is very small in Theorem \ref{thm:main_SN}, we get that the error decreases when $a$ increases, which is 
a property the Samarskii-type schemes do not have in general. However, larger values of $a$ make the 
mesh less dense in the layer.
\end{remark}

\begin{remark}
\label{rem:main_SN+}
Based on Remark \ref{rem:trunc_ySN}, we can prove the error estimate of Theorem \ref{thm:main_SN} with
\[
\nu_1 = \max\left\{ \eps N^{-2} (\ln N)^2,\ \min\{ \eps, N^{-1}\}N^{-1} + \eps^2 N^{-1} + N^{-(a+2)/3} \right\}
\]
if we assume that $a\ge \min\left\{ \frac52, \max\{ 1, a_*\}\right\}$. This assumption is less restrictive, but the
original estimate is better in practice because $N^{-(a+2)/3}\ge N^{-(a+1)/2}$, cf.\  Remark \ref{rem:trunc_ySN}.
\end{remark}

The result of Theorem \ref{thm:main_SN} cannot be improved using the preconditioning approach because of Remark \ref{rem:bar-pre}. This approach
requires truncation-error estimates of the form \eqref{trunc_pre}, but we only have those that follow from the \eqref{trunc_bar}-like form in Lemmas
\ref{lem:trunc_vSN} and \ref{lem:trunc_ySN}. We also keep Remark \ref{rem:barSE} in mind and only discuss below the preconditioning approach on the $S_{1/\eps}$
mesh.

\begin{thm}
\label{thm:main_pre}
Let $b,c,f\in C^4(I)$ and \eqref{bc_cond} hold. 
Further, let $w$ be the solution of the problem \eqref{wprob} and 
let $W^N$ be the solution of the discrete problem \eqref{discr_w} on the mesh $S_{1/\eps}$ with $a>3$
and $N$ sufficiently large independently of $\eps$. Then it holds that
\[
\|w^N - W^N \| \le C \nu_2, 
\]
where
\[
\nu_2 = \max\left\{ \eps|\ln\eps|^3 N^{-2},\ \eps N^{-1} \right\}.
\]
\end{thm}
\begin{proof}
We get from Lemmas \ref{lem:trunc_vSE} and \ref{lem:trunc_ySE} that
\[
|\tau_i[w]| \le C \left\{\begin{array}{l}
(\ln \eps)^2 N^{-2}, \ \ 1\le i \le J-1, \\
\eps N^{-1}, \ \ J\le i \le N-1,
\end{array}\right.
\]
which means that \eqref{trunc_pre} holds with
\[
\tilde\eta_1 = \eps(\ln \eps)^2 N^{-2} \ \text{ and } \ \tilde\eta_2 = \eps N^{-1}.
\]
The result then follows from Theorem \ref{thm:pre}.
\end{proof}

\begin{remark}
\label{rem:main_pre}
If $a\ge 3$ in the above proof, we can show that
\be
\label{main_pre_lin}
\|w^N - W^N \| \le C\eps (\ln\eps)^2 N^{-1}.
\ee
To get this result, we estimate $|\tau_i[w]|$ by combining \eqref{trunc_vSE_lin} and \eqref{trunc_ySE_lin}, and then apply
Theorem \ref{thm:pre}.

The error estimate \eqref{main_pre_lin} is the same as in \cite{VN21}, which is proved there for the upwind scheme on the mesh $S_{1/\eps}$.
However, the present numerical method is simpler than the method in \cite{VN21}. The latter is based on the numerical solution of the 
Neumann problem corresponding to \eqref{1DCD}. The component $v$ of the solution decomposition is completely known for the Neumann problem, 
so $u_0+v$ is used in the same way in \cite{VN21} as the reduced solution $u_0$ in this paper. This leaves $y$ as the 
only unknown component, which is found numerically. 
Although the upwind scheme is simpler than the ASI scheme, the use of $v$ corresponds to the exponential fitting 
that is present in the ASI scheme. What makes the method in \cite{VN21} more complicated is that {\em two}\/ 
numerical approximations of $y$ need to be calculated for the same differential equation with different Neumann 
conditions at $x=0$. Then, a linear combination of these 
two numerical approximations of $y$ is used to form the numerical solution of \eqref{1DCD}.

Since it typically holds that $|\ln\eps| < N$, the result of 
Theorem \ref{thm:main_pre} is practically better than \eqref{main_pre_lin}.
\end{remark}

\begin{remark}
\label{rem:main_pre+}
Similarly to Remark \ref{rem:main_SN+}, we can change the error estimate in Theorem \ref{thm:main_pre} 
if we assume that $a\ge 4$. Then, it can be proved that 
\[
\nu_2 = \max\left\{ \eps|\ln\eps|^3 N^{-2},\ \min\{ \eps, N^{-1}\}N^{-1} + \eps^2 N^{-1} \right\},
\]
which is slightly better than the original estimate.
\end{remark}

\section{Further Discussion}
\label{sec:disc}

Two topics related to the numerical method are discussed in this section.

\subsection{Finding the reduced solution numerically}
\label{sec:u0_numer}
If the reduced solution $u_0$ is not known, a classical numerical method needs to be used to solve the reduced problem \eqref{red} on the 
mesh $S_\lambda$. Let $\hat U^N$ be a numerical approximation $u_0^N$ and suppose
\[
\| u_0^N - \hat U^N\| \le C N^{-k}
\]
fore some $k>0$. In the discrete problem \eqref{discr_w}, we replace $u_0(0)$ with $\hat U^N_0$ and we also need to calculate $u_0''(x_i)$.
The latter is done in \cite[Subsection 4.1]{VNfc} by expressing $u_0'$ from the reduced equation,
\be
\label{u0'}
u_0' = \frac{cu_0-f}{b},
\ee
and then using $\bar D' u'_0(x_i)$ (see \eqref{D'cent}) and $u_0(x_i)\approx \hat U^N_i$ to approximate $u_0''(x_i)$.
We discuss a more accurate alternative here, in which we differentiate \eqref{u0'} and reuse it to express $u_0''$ in terms of $u_0$ and the 
functions $b^{(j)}$, $c^{(j)}$, and $f^{(j)}$, $j=0,1$.
When $u_0(x_i)$ is replaced with $\hat U^N_i$, it follows that $u_0''(x_i)$ is approximated with the same $\Oh\left(N^{-k}\right)$-accuracy 
as $u_0(x_i)$. Let $\hat W^N$ stand for the unique solution of the discrete problem \eqref{discr_w} with the right-hand side modified in the 
way just described.

We now estimate the quantity $\| W^N - \hat W^N\|$ which shows how much the approximation of $u_0$ contributes to the error estimates from 
the preceding section. Since the operator $L^N$ satisfies the maximum principle, we have
\begin{eqnarray*}
 \| W^N - \hat W^N\| & \le & C\left( | W^N_0 - \hat W^N_0 | + \| L^N W^N - L^N \hat W^N\| \right) \\  
 & \le & C\left( N^{-k} + \eps N^{-k}\right) \\
 & \le & C N^{-k}.
\end{eqnarray*}

\begin{remark}
\label{rem:RKrate}
If $k$ is sufficiently large, the accuracy of our numerical method on the mesh $S_N$ is practically preserved, cf.\ Theorem 
\ref{thm:main_SN}. On the mesh $S_{1/\eps}$, however, the term $\Oh\left(N^{-k}\right)$ may influence the accuracy,  
see Theorem \ref{thm:main_pre}. 
\end{remark}

\subsection{The case of the linear reduced solution}
\label{sec:u0lin}

If the reduced solution $u_0$ is linear, there is no need to decompose $u$. We can discretize the original problem \eqref{1DCD} directly 
instead of using the problem \eqref{wprob}. This is because $\tau_i[u_0]=0$, so that $\tau_i[u] = \tau_i[w]$, $1\le i \le N-1$. 
Moreover, when $u_0$ is linear, the derivatives of the component $y$ can be estimated more sharply than in \eqref{dery},
\be
\label{dery_u0lin}
	|y^{(j)}(x)| \le C\eps^{1-j}e^{-\beta x/\eps}, \ \ j=0,1,2,3,4, \ \ x\in [0,1].
\ee
The derivative estimates in \eqref{dery} are proved in \cite[Theorem 2]{VN21} indirectly via a Neumann problem corresponding to 
\eqref{1DCD}. If $u_0$ is linear, the same technique gives the estimates in \eqref{dery_u0lin}.
Using the sharper derivative estimates in \eqref{dery_u0lin}, we can improve the error estimates of Theorems \ref{thm:main_SN} 
and \ref{thm:main_pre}. The improvement is more significant on $S_N$ than on $S_{1/\eps}$.

\begin{thm}
\label{thm:main_SN_u0lin}
Let $b,c,f\in C^4(I)$ and \eqref{bc_cond} hold, and suppose the reduced problem \eqref{red} has a linear solution.
Further, let $u$ be the solution of the problem \eqref{1DCD} and 
let $U^N$ be the solution of the discrete problem \eqref{discr_u} on the mesh $S_N$ with $a\ge a_*$.
Then it holds that
\[
\|u^N - U^N \| \le C\max\left\{ \eps N^{-2}(\ln N)^2,\ N^{-a}\right\}.
\]
\end{thm}
\begin{proof}
We only need to show the improved estimates for $|\tau_i[y]|$. When \eqref{dery_u0lin} is used instead of \eqref{dery}, it is easy to verify 
that the proof technique of Lemma \ref{lem:trunc_ySN} gives
\[
|\tau_i[y]| \le C\left\{\begin{array}{ll}
N^{-2}(\ln N)^2 e^{-\beta x_i/\eps}, & 1\le i \le J-1, \\
N^{-a}, & J\le i \le N-1.
\end{array}\right.
\]
The case $i=J,J+1$ is now much simpler as it suffices to use \eqref{trunc0},
\[
|\tau_i[y]| \le C\left( \eps \|y''\|_{[x_{i-1},x_{i+1}]} + \| y'\|_{[x_i,x_{i+1}]}\right) \le C e^{-\beta x_{i-1}/\eps} \le C N^{-a}.
\]
Note that the above estimates of $|\tau_i[y]|$ do not impose any requirement on the mesh parameter $a$. The assumption $a\ge a_*$ is only 
needed to get the estimates of $|\tau_i[v]|$ like in Lemma \ref{lem:trunc_vSN}.

Therefore, we have
\[
|\tau_i[u]| \le C \left\{\begin{array}{ll}
\eps N^{-2}(\ln N)^2 \eps^{-1}e^{-\beta x_i/\eps}, & 1\le i \le J-1, \\
 N^{-a}, & J\le i \le N-1.
\end{array}\right.
\]
Then, the result of the theorem follows from Theorem \ref{thm:bar}.
\end{proof}

The following counterpart of Theorem \ref{thm:main_pre} also holds.

\begin{thm}
\label{thm:main_SE_u0lin}
Let $b,c,f\in C^4(I)$ and \eqref{bc_cond} hold, and suppose the reduced problem \eqref{red} has a linear solution.
Further, let $u$ be the solution of the problem \eqref{1DCD} and 
let $U^N$ be the solution of the discrete problem \eqref{discr_u} on the mesh $S_{1/\eps}$ with $a\ge 4$. 
Then it holds that
\[
\|u^N - U^N \| \le C \max\left\{ \eps|\ln\eps|^3 N^{-2},\ \eps N^{-2} + \eps^2 N^{-1} \right\},
\]
provided $N$ is sufficiently large independently of $\eps$.
\end{thm}
\begin{proof}
We follow the steps in the proof of Lemma \ref{lem:trunc_ySE} in the situation when \eqref{dery_u0lin} holds.
The estimates of $|\tau_i[y]|$ do not change for $1\le i \le J-1$, but for the other values of $i$, we get
\[
|\tau_J[y]| \le C\left( \eps N^{-2} + \eps^2 N^{-1} \right)
\]
and
\[
|\tau_i[y]| \le C\eps N^{-2}, \ \ J+1\le i \le N-1.
\]
When this is combined with Lemma \ref{lem:trunc_vSE}, it follows that
\[
|\tau_i[u]| \le C \left\{\begin{array}{ll}
(\ln\eps)^2 N^{-2}, & 1\le i \le J-1, \\
\left( \eps N^{-2} + \eps^2 N^{-1} \right) , & J\le i \le N-1.
\end{array}\right.
\]
We then use Theorem \ref{thm:pre} to complete the proof.
\end{proof}

\section{Numerical Results}
\label{sec:numer}

Our first test problem is from~\cite[page 86]{Linss10},
\be
\label{ex1u}
-\eps u'' -u' +2u =e^{x-1}, \ \ x\in (0,1), \ \  u(0)=u(1)=0.
\ee
The reduced solution of the problem \eqref{ex1u} is known, $u_0(x) = e^{x-1}-e^{2(x-1)}$, so we can form the component $w=u-u_0$ which solves
\be
\label{ex1}
-\eps w'' -w' +2w = \eps u_0''(x),  \ \ x\in (0,1), \ \  w(0)=-u_0(0), \ \ w(1)=0.
\ee
The exact solution of \eqref{ex1u} is also known, and we can calculate the error of our numerical method directly,
\[
E^N_\eps = \left\| w^N-W^N\right\|,
\]
as well as the numerical rate of convergence, 
\[
r^N_\eps= \frac{\ln E^{N/2}_\eps - \ln E^{N}_\eps}{\ln 2}.
\]

In the results presented below, the parameters $\beta$ and $Q$ are fixed to $\beta=0.8$ and $Q=\frac12$. 
We also experimented with other values of $\beta\in (0.8,1)$ and $Q\in\left[\frac12, 1\right)$, and the results were 
similar. As for the parameter $a$, we show how it influences the error in Figure \ref{fig:a}. Elsewhere, it is set to $a=3$.

\begin{table}\small
	\centering
	\begin{tabular}{c||c|c|c|c|c|c}
		$\eps$ &  $N=2^5$   & $N=2^6$    &      $N=2^7$    & $N=2^8$    &  $N=2^9$  &$N=2^{10}$ \\
		\hline\hline
$10^{-1}$ & 1.27E-04 & 3.18E-05 & 7.95E-06 & 1.99E-06 & 4.97E-07 & 1.24E-07 \\
 &  & 2.00 & 2.00 & 2.00 & 2.00 & 2.00  \\ \hline
$10^{-2}$  &  4.13E-04 & 9.08E-05 & 2.66E-05 & 5.87E-06 & 1.33E-06 & 3.08E-07 \\
 &  & 2.19 & 1.77 & 2.18 & 2.14 & 2.12    \\ \hline
$10^{-3}$ &  6.88E-05 & 3.02E-05 & 1.18E-05 & 5.41E-06 & 1.87E-06 & 5.29E-07\\
 &  & 1.19 & 1.35 & 1.13 & 1.53 & 1.82    \\ \hline
$10^{-4}$ &  5.25E-06 & 2.71E-06 & 1.51E-06 & 6.81E-07 & 3.33E-07 & 1.58E-07\\
  &  & 0.96 & 0.84 & 1.15 & 1.03 & 1.08    \\ \hline
 $10^{-5}$ &  7.19E-07 & 3.27E-07 & 1.39E-07 & 7.35E-08 & 3.50E-08 & 1.77E-08\\
  &  & 1.14 & 1.24 & 0.91 & 1.07 & 0.98    \\ \hline
$10^{-6}$ &  5.51E-07 & 5.12E-08 & 1.39E-08 & 7.37E-08 & 3.61E-09 & 1.76E-09\\
  &  & 3.43 & 1.88 & 0.91 & 1.03 & 1.04   \\ \hline
 $10^{-7}$ &  2.37E-07 & 4.04E-08 & 2.83E-09 & 7.00E-10 & 3.52E-10 & 1.76E-10\\ 
  &  & 2.55 & 3.83 & 2.02 & 0.99 & 1.00   \\ \hline
$10^{-8}$ &  5.28E-07 & 2.42E-08 & 2.26E-09 & 2.47E-10 & 3.52E-11 & 1.76E-11\\
 &  & 4.45 & 3.42 & 3.19 & 2.81 & 1.00    \\ \hline
 $10^{-9}$ &  5.27E-07 & 3.92E-08 & 2.92E-09 & 2.20E-10 & 1.62E-11 & 1.90E-12\\
  &  & 3.75 & 3.75 & 3.73 & 3.76 & 3.09  
\end{tabular}
	\caption{The maximum errors $E^N_\eps$ and the corresponding convergence rates $r^N_\eps$ on $S_N$ for the problem \eqref{ex1}.}
 	\label{tab:ex1_SN}
\end{table}

Table \ref{tab:ex1_SN} illustrates the efficiency of our numerical method. The accuracy is very high, but the convergence 
rates look chaotic. The error estimate $\nu_1$ in Theorem \ref{thm:main_SN} does not match the behavior of the errors in Table \ref{tab:ex1_SN} well. 
This is why we consider the quantity
\[
F^N_\eps := \max\left\{ \eps N^{-2}(\ln N)^2,\ \min\{\eps, N^{-1}\}N^{-1} + N^{-a}\right\},
\]
which replaces the $\nu_1$-terms  $\eps N^{-1}$ and $N^{-(a+1)/2}$ with the terms that are present in the estimate of $|\tau_i[w]|$ for $J+2\le i \le N-1$.
The replaced terms result from the proof technique and do not seem to represent the real situation. They stem from the estimate of $|\tau_i[y]|$ at two points only, 
viz., $x_J$ and $x_{J+1}$. The agreement between the values of $F^N_\eps$ and $E^N_\eps$ in Table \ref{tab:ex1_SN} is represented graphically in Figure \ref{fig:T1}.
The values match well when $\eps$ is greater, $\eps = 10^{-j}$, $j=1,2,3,4$, as exemplified by $\eps = 10^{-3}$.  The match is also good for very small values of $\eps$, 
like $\eps = 10^{-9}$, but not perfect for the intermediate values from $\eps=10^{-5}$ to $\eps=10^{-7}$, as shown by $\eps=10^{-6}$ in Figure \ref{fig:T1}.
However, $F^N_\eps$ gives a better overall representation of the errors than $\nu_1$.

\begin{figure}[!ht]
\centering
\includegraphics[width=0.7\textwidth]{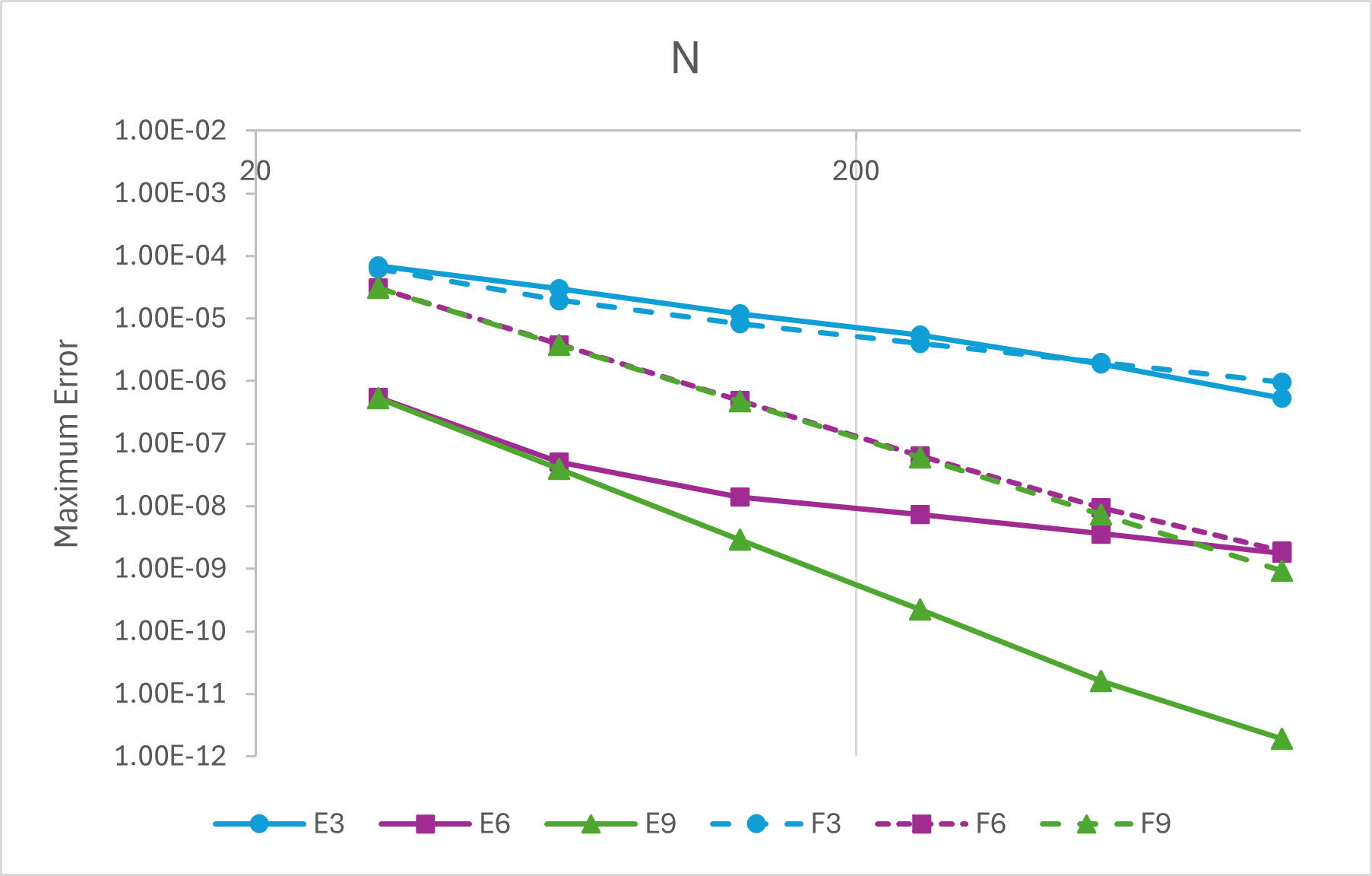}
\caption{Graphical comparison of the errors in Table \ref{tab:ex1_SN} and $F^N_\eps$. The labels E$j$ and F$j$ stand for the values of 
$E^N_\eps$ and, respectively, $F^N_\eps$ for $\eps=10^{-j}$.}
\label{fig:T1}
\end{figure}

We now discuss the numerical rates of convergence in Table \ref{tab:ex1_SN}. The mesh is uniform when $\eps= 0.1$, and the rate is 2 since $\eps$ is large, \cite{KT78}. 
For other values of $\eps$, the mesh is nonuniform, and the behavior of the convergence rates can be explained by different terms that dominate the error
in dependence on the values of $\eps$ and $N$. It turns out the term $\eps N^{-2} (\ln N)^2$ is always less than the other terms of $F^N_\eps$.
This leaves the terms within $\min\{\eps, N^{-1}\}N^{-1} + N^{-a}$ for consideration. Convergence rates close to 1 can be explained by 
the term $\eps N^{-1}$ dominating the error. If the dominant term is $N^{-2}$, the rates are close to 2, and if the dominant term is $N^{-a}= N^{-3}$, the rates 
are around 3.

 \begin{table}[!ht]\small
	\centering
	\begin{tabular}{c||c|c|c|c|c|c}
		$\eps$ & $N=2^5$   & $N=2^6$    &      $N=2^7$    & $N=2^8$    &  $N=2^9$  &$N=2^{10}$ \\
		\hline\hline
$10^{-1}$ & 3.53E-04 & 8.84E-05 & 2.21E-05 & 5.53E-06 & 1.38E-06 & 3.46E-07\\
	&  & 2.00 &    2.00 &     2.00 &     2.00 &   2.00      \\ \hline
$10^{-2}$ &  1.08E-02 & 2.53E-03 & 7.13E-04 & 1.61E-04 & 3.69E-05 & 8.54E-06\\
  &  & 2.09 & 1.83 & 2.15 & 2.12 & 2.11   \\ \hline
$10^{-3}$ & 1.81E-02 & 8.05E-03 & 3.29E-03 & 1.47E-03 & 5.12E-04 & 1.45E-04\\
   &  & 1.17 & 1.29 & 1.16 & 1.53 & 1.82    \\ \hline
$10^{-4}$ & 1.47E-02 & 7.54E-03 & 4.09E-03 & 1.89E-03 & 9.22E-04 & 4.37E-04\\
	& & 0.96 & 0.88 & 1.11 & 1.03 & 1.08      \\ \hline
$10^{-5}$ & 1.90E-02 & 8.72E-03 & 3.85E-03 & 2.01E-03 & 9.69E-04 & 4.88E-04\\
	&  & 1.12 & 1.18 & 0.94 & 1.05 & 0.99      \\ \hline
$10^{-6}$ & 1.48E-02 & 7.60E-03 & 3.85E-03 & 2.02E-03 & 9.93E-04 & 4.88E-04\\
	&  & 0.96 & 0.98 & 0.94 & 1.02 & 1.03    \\ \hline
 $10^{-7}$ & 1.90E-02 & 7.60E-03 & 4.15E-03 & 1.94E-03 & 9.74E-04 & 4.88E-04\\
&  & 1.32 & 0.87 & 1.09 & 0.99 & 1.00    \\ \hline
$10^{-8}$ & 1.48E-02 & 8.72E-03 & 4.15E-03 & 1.94E-03 & 9.75E-04 & 4.88E-04\\ 
&  &   0.76 & 1.07 & 1.09 & 0.99 & 1.00   \\ \hline
$10^{-9}$ & 1.48E-02 & 7.60E-03 & 3.86E-03 & 1.94E-03 & 9.93E-04 & 4.93E-04\\
&  &  0.96 & 0.98 & 0.99 & 0.97 & 1.01 
\end{tabular}
\caption{The maximum errors $E^{*,N}_\eps$ and the corresponding convergence rates on $S_N$ for the problem \eqref{ex1u}.}  
\label{tab:ex1_SN_comp}
\end{table}

The results in Table \ref{tab:ex1_SN} should be compared to those in Table \ref{tab:ex1_SN_comp}, which were obtained 
when the ASI scheme is used to directly discretize \eqref{ex1u} on the mesh $S_N$. In this situation, the maximum error is
\[
E^{*,N}_\eps := \left\| u^N-U^N\right\|,
\]
where $U^N$ is the solution to the discretization \eqref{discr_u} of the problem \eqref{ex1u}. All errors in Table \ref{tab:ex1_SN} are less 
than the corresponding ones in Table \ref{tab:ex1_SN_comp}, more significantly so for smaller values of $\eps$. 
A graphical comparison of the errors is presented in Figure \ref{fig:E-Est} for $\eps=10^{-8}$.
Also included in Figure \ref{fig:E-Est} are the errors of the Samarskii scheme on the mesh $S_N$ for the problem \eqref{ex1}. 
We can conclude that it is vital for the high accuracy of our method to use both $u_0$ and the ASI scheme. If either 
of the two is omitted, the results are considerably worse. By the way, it is interesting 
to note that when $\eps$ is smaller, it is better to use $u_0$ than the ASI scheme. This is so because the Samarskii scheme is 
almost-second-order accurate when applied to the problem \eqref{ex1}, see Theorem \ref{thm:main_old}, whereas the order of accuracy of the 
ASI scheme applied to problem \eqref{ex1u} is 1, \cite{VNfc}.

\begin{figure}[!ht]
\centering
\includegraphics[width=0.7\textwidth]{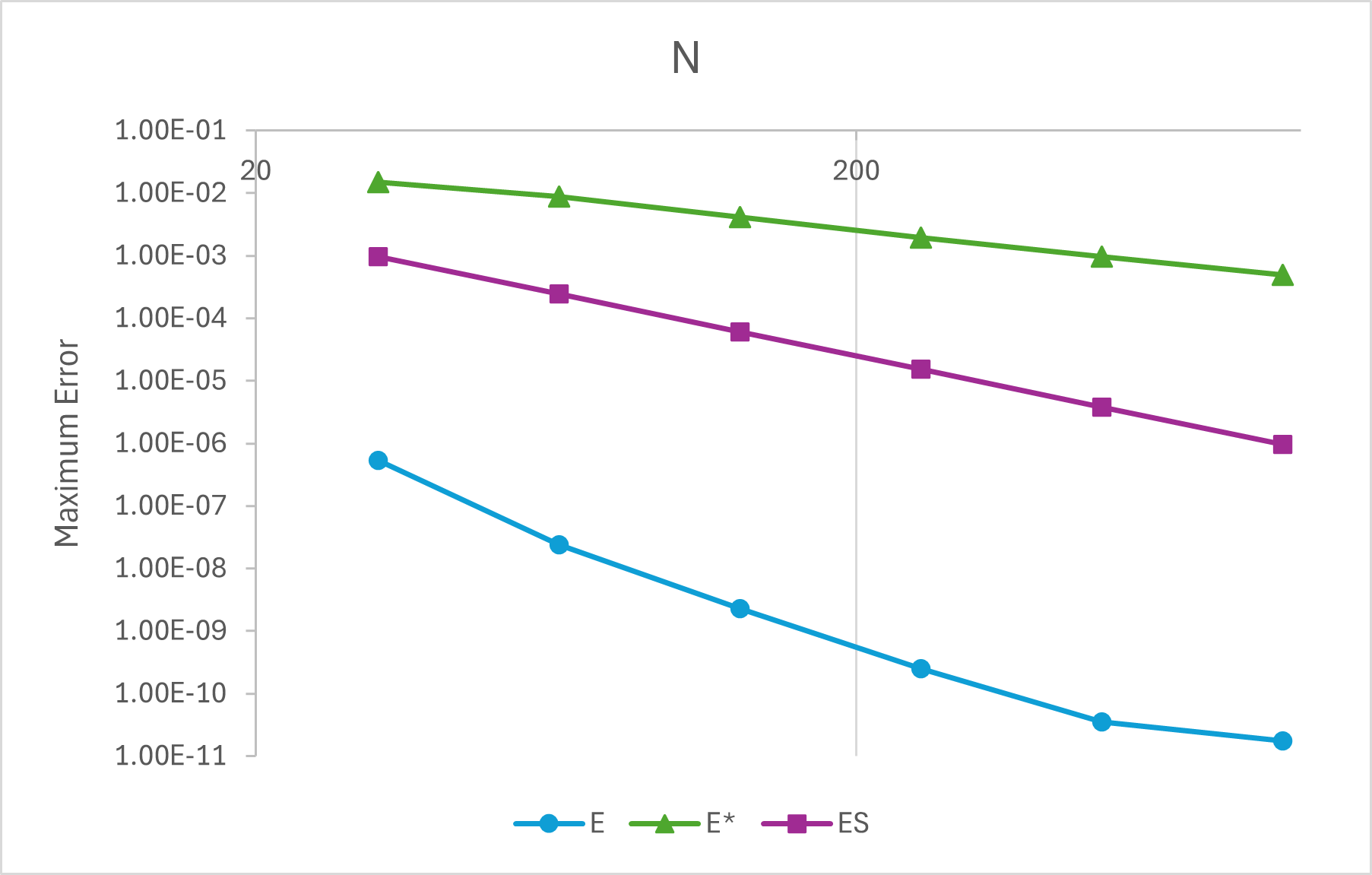}
\caption{Graphical comparison of the errors in Tables \ref{tab:ex1_SN} and \ref{tab:ex1_SN_comp}, and the errors of the 
Samarskii scheme on the mesh $S_N$ for the problem \eqref{ex1} with $\eps=10^{-8}$. 
The labels E, E*, and ES stand for the values of $E^N_\eps$, $E^{*,N}_\eps$, and the Samarskii-scheme errors, respectively.}
\label{fig:E-Est}
\end{figure}

As for the convergence rates, those in Table \ref{tab:ex1_SN} are either close to or better than the corresponding rates in 
Table \ref{tab:ex1_SN_comp}. However, the latter rates show more regularity: for greater values of $\eps$ (when the mesh is uniform or 
closer to a uniform one) the rates are equal or approximately equal to 2, 
and then they drop to 1 as $\eps$ decreases. The same change in the convergence rate as $\eps$ decreases is also present on the uniform 
mesh, \cite{KT78}. This behavior on the mesh $S_N$ has been observed in the numerical experiments in \cite{VNfc}. 

We conclude this part of the numerical experiments with the ASI scheme on the mesh $S_N$ by exploring how the mesh parameter $a$ influences the 
error. Figure \ref{fig:a} shows a comparison of errors for different values of $a$ and $\eps=10^{-8}$. According to 
Theorem \ref{thm:main_SN}, the values of $a\ge 3$ are theoretically safe, but we see that even $a=1$ and $a=2$ give acceptable results.
In general, the errors are smaller if $a$ is greater, but as $N$ increases, the term $N^{-a}$ becomes too small and the term
$\eps N^{-1}$ starts dominating the error. This is why for $a\ge 3$, the errors eventually become the same and the convergence rate becomes 
1 the sooner the greater the value of $a$. For $a=1$ and $a=2$, the rates are determined by the error term $N^{-a}$.

\begin{figure}[!ht]
\centering
\includegraphics[width=0.7\textwidth]{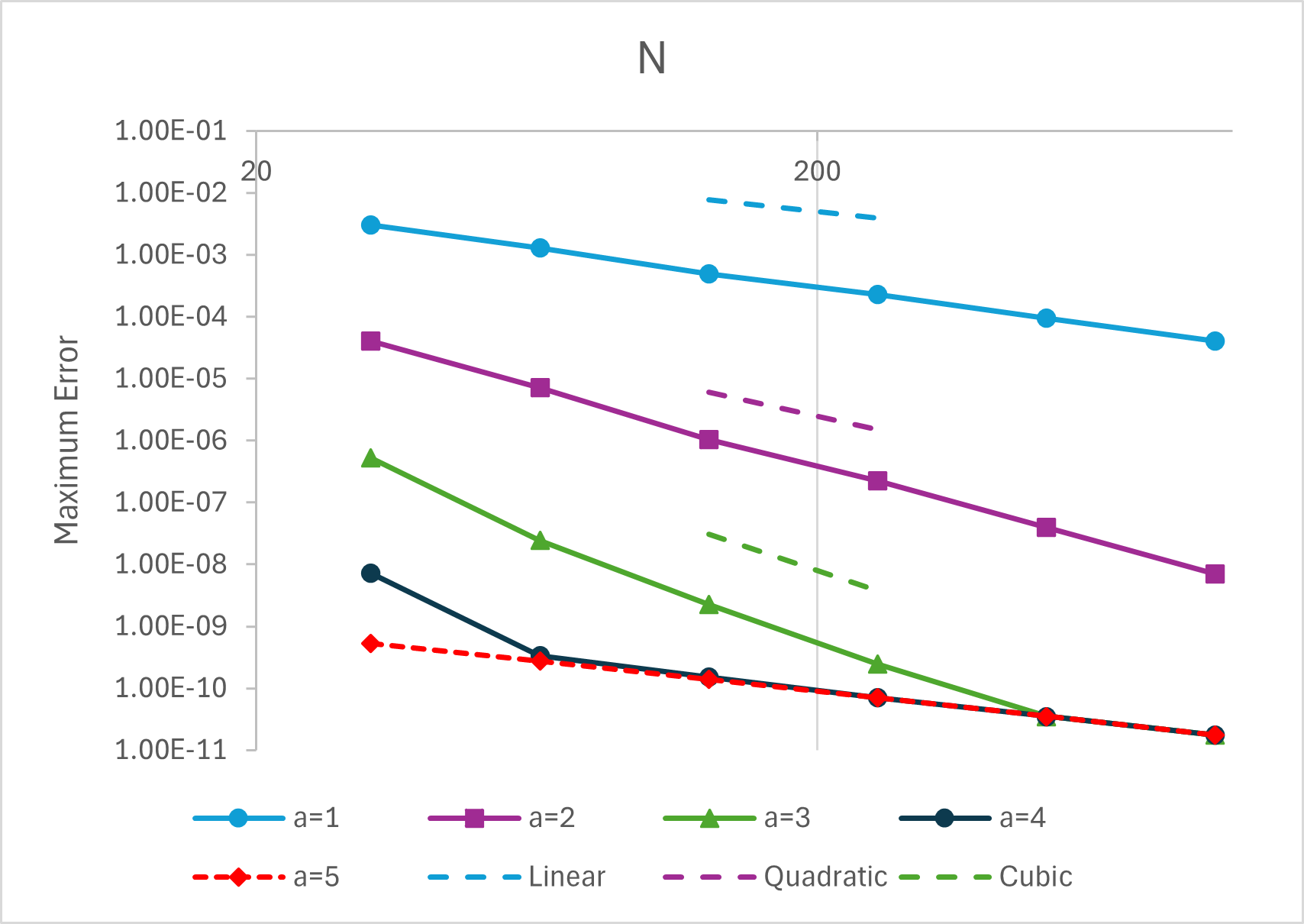}
\caption{Graphical comparison of the errors $E^N_\eps$ of the ASI scheme for $\eps=10^{-8}$, obtained when solving the problem \eqref{ex1} 
on the mesh $S_N$ with different values of $a$. The slopes of linear, quadratic, and cubic rates are included for comparison.}
\label{fig:a}
\end{figure}

We now turn to the mesh $S_{1/\eps}$. The results for the problem \eqref{ex1} are presented in Table \ref{tab:ex1_SE}, 
from which $\eps=0.1$ is omitted since the mesh is uniform and the results are the same as in Table \ref{tab:ex1_SN}. 
Although Theorem \ref{thm:main_pre} requires $a>3$, the method works well for $a=3$. We can see that, as $N$ increases, the rate of
convergence is around 2 for $\eps=10^{-2}$, but gets closer to 1 when $\eps$ is smaller. This agrees with the theoretical result in 
Theorem \ref{thm:main_pre}.
The errors also generally decrease with $\eps$. This decrease is moderate or even absent when $\eps$ changes from $10^{-2}$ to $10^{-3}$,
but as $\eps$ continues to decrease, the errors become smaller at a rate which is 1 or close to 1. This is better than what  
Theorem \ref{thm:main_pre} suggests, and even better than $\eps|\ln\eps|$ (this quantity reaches
the greatest rate of decrease when $\eps$ changes from $10^{-8}$ to $10^{-9}$, and this rate is 0.95).
Therefore, for smaller $\eps$ values, the behavior of the errors in Table \ref{tab:ex1_SE} looks like $\Oh\left(\eps N^{-1}\right)$.

\begin{table}[!ht]\small
	\centering
	\begin{tabular}{c||c|c|c|c|c|c}
		$\eps$ &  $N=2^5$   & $N=2^6$    &      $N=2^7$    & $N=2^8$    &  $N=2^9$  &$N=2^{10}$ \\
		\hline\hline
$10^{-2}$  & 2.72E-04 & 8.76E-05 & 2.73E-05 & 6.11E-06 & 1.54E-06 & 3.87E-07 \\
& & 1.64 & 1.68 & 2.16 & 1.99 & 1.99  \\ \hline
$10^{-3}$ &  4.96E-05 & 2.48E-05 & 1.17E-05 & 5.37E-06 & 1.78E-06 & 5.14E-07\\
 &  & 1.00 & 1.08 & 1.13 & 1.59 & 1.79  \\ \hline
$10^{-4}$ &  5.24E-06 & 3.24E-06 & 1.37E-06 & 6.80E-07 & 3.33E-07 & 1.60E-07\\
  &  & 0.69 & 1.24 & 1.01 & 1.03 & 1.06  \\ \hline
 $10^{-5}$ &  5.27E-07 & 3.27E-07 & 1.39E-07 & 6.98E-08 & 3.50E-08 & 1.77E-08\\
  &  &  0.69 & 1.24 & 0.99 & 1.00 & 0.98  \\ \hline
$10^{-6}$ &  5.27E-08 & 3.27E-08 & 1.39E-08 & 7.37E-09 & 3.61E-09 & 1.78E-09\\
  &  & 0.69 & 1.24 & 0.91 & 1.03 & 1.02  \\ \hline
 $10^{-7}$ & 7.20E-09 & 3.27E-09 & 1.53E-09 & 7.00E-10 & 3.52E-10 & 1.76E-10 \\
  &  &  1.14 & 1.09 & 1.13 & 0.99 & 1.00  \\ \hline
$10^{-8}$ &  5.27E-10 & 3.27E-10 & 1.53E-10 & 7.37E-11 & 3.61E-11 & 1.79E-11\\
 &  &  0.69 & 1.09 & 1.06 & 1.03 & 1.02    \\ \hline
 $10^{-9}$ & 5.27E-11 & 2.73E-11 & 1.53E-11 & 7.37E-12 & 3.52E-12 & 1.76E-12 \\
 &  & 0.95 & 0.83 & 1.06 & 1.07 & 1.00 
\end{tabular}
	\caption{The maximum errors $E^N_\eps$ and the corresponding convergence rates $r^N_\eps$ on $S_{1/\eps}$ for the problem 
\eqref{ex1}.}
 	\label{tab:ex1_SE}
\end{table}

When we compare Tables \ref{tab:ex1_SN} and \ref{tab:ex1_SE}, we can see that the rates in the latter are less intricate by far.
Also, the smaller the $\eps$ value, the lower the errors in Table \ref{tab:ex1_SE} than in Table \ref{tab:ex1_SN}. 
This can be explained by the $\eps$-factor in the error estimate on the mesh $S_{1/\eps}$, which is not present in all error terms on the mesh $S_N$.
A graphical comparison of the results in Tables \ref{tab:ex1_SN} and \ref{tab:ex1_SE} is given in Figure \ref{fig:SN-SE}.

\begin{figure}[!ht]
\centering
\includegraphics[width=0.7\textwidth]{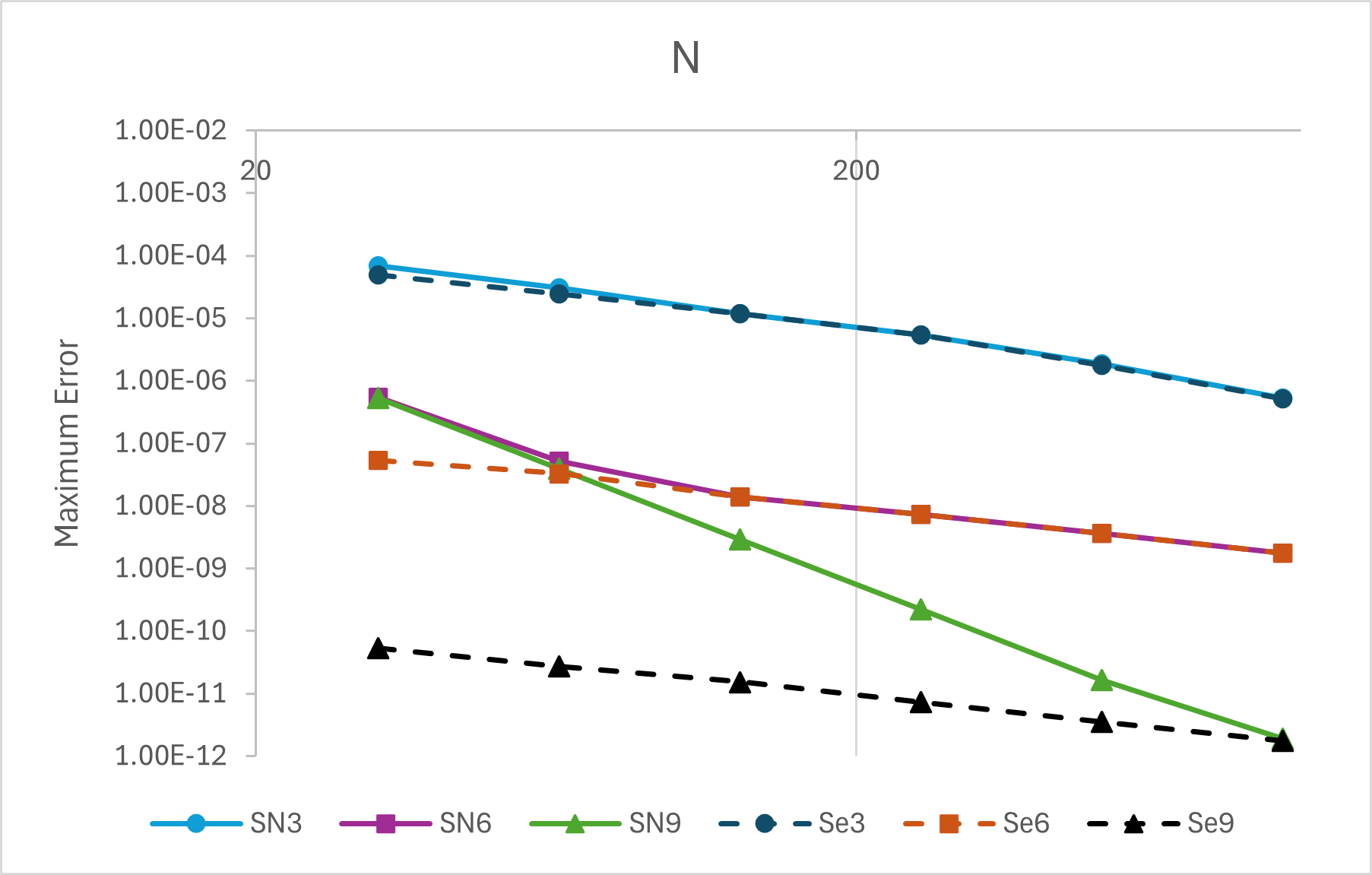}
\caption{Graphical comparison of the errors $E^N_\eps$ of the ASI scheme on $S_N$ and $S_{1/\eps}$ for
the problem \eqref{ex1}.
The labels SN$j$ and Se$j$ stand for the values of the errors on $S_N$ and, respectively, $S_{1/\eps}$ for $\eps=10^{-j}$.}
\label{fig:SN-SE}
\end{figure}

The combined use of $u_0$ and the ASI scheme is important for the high accuracy of our method on the mesh $S_{1/\eps}$ as well.
When the ASI scheme on the mesh $S_{1/\eps}$ is applied directly to the problem \eqref{ex1u}, the results are very similar to those in
Table \ref{tab:ex1_SN_comp}. As for other Samarskii-type schemes applied to the problem \eqref{ex1}, they perform worse on
$S_{1/\eps}$ than on $S_N$. Although the errors on $S_{1/\eps}$ are acceptable (the errors for $N=1024$ have magnitudes from about 
$10^{-5}$ to $10^{-4}$), they are not uniform in $\eps$ since they slightly increase as $\eps$ decreases (the errors very likely contain an 
$|\ln\eps|$-factor).

We next illustrate the situation when the reduced solution $u_0$ is calculated numerically.  We used the 4th-order Runge-Kutta method on $S_\lambda$ to solve the 
reduced problem corresponding to \eqref{ex1u} and applied the procedure described in Subsection \ref{sec:u0_numer}. As predicted in Remark \ref{rem:RKrate},
the results on $S_N$ are almost identical to those in Table \ref{tab:ex1_SN}, so there is no need to present them here. Only in a few cases of larger $N$-values is the
error somewhat greater than in Table \ref{tab:ex1_SN}. The difference is more pronounced on $S_{1/\eps}$ when both $\eps$ and $N$ are smaller. 
This is shown in Figure \ref{fig:Se-S0}, where the errors $E^N_\eps$ and
\[
\hat E^N_\eps : = \| w^N - \hat W^N\|
\]
are compared graphically. There is no essential difference between $E^N_\eps$ and $\hat E^N_\eps$ for greater values of $\eps$, but $\eps=10^{-9}$
shows that $E^N_\eps$ is less when $N$ is smaller. The term $N^{-4}$ dominates the error initially, but the fourth-rate convergence is only evident for
$N=256, 512$. Finally, when $N=512, 1024$, the term $N^{-4}$ is no longer dominant, so $E^N_\eps$ and $\hat E^N_\eps$ become identical.

\begin{figure}[!ht]
\centering
\includegraphics[width=0.7\textwidth]{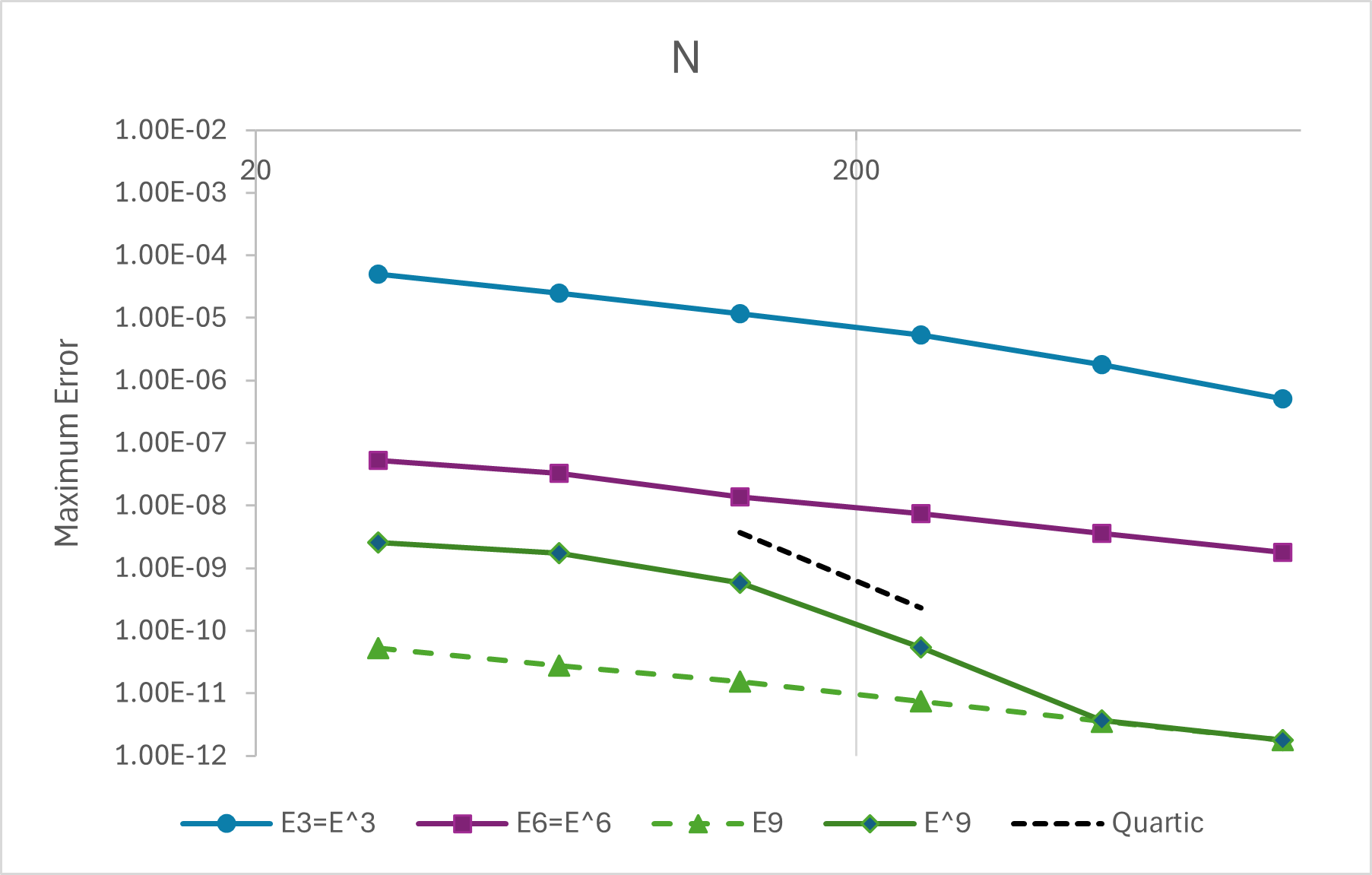}
\caption{Graphical comparison of the errors $E^N_\eps$ and $\hat E^N_\eps$ of the ASI scheme on $S_{1/\eps}$ for
the problem \eqref{ex1}.
The labels E$j$ and E{}\^{$\! j$} stand for the values of  $E^N_\eps$ and, respectively, $\hat E^N_\eps$ for $\eps=10^{-j}$.
The slope of the quartic rate is included for comparison.}
\label{fig:Se-S0}
\end{figure}

We finish our numerical experiments by illustrating Subsection \ref{sec:u0lin}. Like in \cite{VNfc}, we consider a problem 
taken from \cite{HMMR95, Lenf00},
\be
\label{ex2}
-\eps u'' -(1-x)u' + 2u = 3(1-x), \ \  x\in (0,1), \ \ u(0)=2, \ \ u(1)=0.
\ee
Its exact solution is known, $u(x)=(1-x)\left[1+e^{-x(2-x)/(2\eps)}\right]$,
and the reduced solution is linear, $u_0(x)=1-x$. In this problem, the function $b$ does not satisfy
the condition in \eqref{bc_cond}, but $b>0$ on $[0,1)$ and $u$ has an exponential layer at $x=0$. 
Tables \ref{tab:ex2_SN} and \ref{tab:ex2_SE} show the results for the ASI scheme applied directly to \eqref{ex2} on $S_N$ and $S_{1/\eps}$ 
respectively. Although $a=3$ does not meet the requirements of Theorems \ref{thm:main_SN_u0lin} and \ref{thm:main_SE_u0lin}, the results behave as these theorems
suggest. In particular, the rates of convergence are better than the corresponding ones in Tables \ref{tab:ex1_SN} and \ref{tab:ex1_SE} because the results of
Theorems \ref{thm:main_SN_u0lin} and \ref{thm:main_SE_u0lin} are better than those of Theorems \ref{thm:main_SN} and \ref{thm:main_pre} respectively.
All this is accomplished without the decomposition $u=u_0+w$.

\begin{table}\small
	\centering
	\begin{tabular}{c||c|c|c|c|c|c}
		$\eps$ &  $N=2^5$   & $N=2^6$    &      $N=2^7$    & $N=2^8$    &  $N=2^9$  &$N=2^{10}$ \\
		\hline\hline
$10^{-1}$ & 4.23E-04 & 1.06E-04 & 2.65E-05 & 6.64E-06 & 1.66E-06 & 4.15E-07\\
 &  & 2.00 & 2.00 & 2.00 & 2.00 & 2.00  \\ \hline
$10^{-2}$  & 2.05E-04 & 7.60E-05 & 2.57E-05 & 8.48E-06 & 2.69E-06 & 8.29E-07 \\
 &  & 1.43 & 1.56 & 1.60 & 1.66 & 1.70   \\ \hline
$10^{-3}$ &  1.91E-05 & 7.22E-06 & 2.46E-06 & 8.12E-07 & 2.57E-07 & 7.94E-08 \\
 &  &  1.41 & 1.56 & 1.60 & 1.66 & 1.70  \\ \hline
$10^{-4}$ &  2.28E-06 & 6.20E-07 & 2.39E-07 & 8.03E-08 & 2.56E-08 & 7.90E-09\\
  &  &  1.88 & 1.37 & 1.57 & 1.65 & 1.69   \\ \hline
 $10^{-5}$ &  1.01E-06 & 1.04E-07 & 1.73E-08 & 7.59E-09 & 2.52E-09 & 7.88E-10\\
  &  & 3.28 & 2.59 & 1.18 & 1.59 & 1.68    \\ \hline
$10^{-6}$ &  2.27E-06 & 1.69E-07 & 1.25E-08 & 7.92E-10 & 2.17E-10 & 7.62E-11 \\
  &  &  3.75 & 3.75 & 3.98 & 1.87 & 1.51   \\ \hline
 $10^{-7}$ &  1.01E-06 & 1.69E-07 & 9.43E-09 & 9.31E-10 & 6.92E-11 & 5.58E-12\\
  &  & 2.58 & 4.16 & 3.34 & 3.75 & 3.63    \\ \hline
$10^{-8}$ & 2.27E-06 & 1.04E-07 & 9.43E-09 & 9.31E-10 & 6.93E-11 & 5.15E-12\\
 &  &  4.45 & 3.46 & 3.34 & 3.75 & 3.75   \\ \hline
 $10^{-9}$ &  2.27E-06 & 1.69E-07 & 1.25E-08 & 9.31E-10 & 6.32E-11 & 4.89E-12\\
  &  &   3.75 & 3.75 & 3.75 & 3.88 & 3.69 
\end{tabular}
	\caption{The maximum errors $E^N_\eps$ and the corresponding convergence rates $r^N_\eps$ on $S_N$ for the problem \eqref{ex2}.}
 	\label{tab:ex2_SN}
\end{table}

\begin{table}\small
	\centering
	\begin{tabular}{c||c|c|c|c|c|c}
		$\eps$ &  $N=2^5$   & $N=2^6$    &      $N=2^7$    & $N=2^8$    &  $N=2^9$  &$N=2^{10}$ \\
		\hline\hline
$10^{-2}$  & 3.66E-04 & 9.29E-05 & 2.33E-05 & 5.84E-06 & 1.46E-06 & 3.66E-07\\
 &  & 1.98 & 1.99 & 2.00 & 2.00 & 2.00    \\ \hline
$10^{-3}$ &  6.75E-05 & 1.96E-05 & 4.93E-06 & 1.26E-06 & 3.15E-07 & 7.89E-08\\
 &  &  1.79 & 1.99 & 1.97 & 2.00 & 2.00    \\ \hline
$10^{-4}$ &  9.01E-06 & 3.50E-06 & 8.86E-07 & 2.22E-07 & 5.57E-08 & 1.40E-08\\
  &  &   1.37 & 1.98 & 1.99 & 2.00 & 2.00  \\ \hline
 $10^{-5}$ &  9.88E-07 & 5.16E-07 & 1.32E-07 & 3.48E-08 & 8.71E-09 & 2.18E-09\\
  &  &  0.94 & 1.97 & 1.92 & 2.00 & 2.00    \\ \hline
$10^{-6}$ &  9.55E-08 & 6.72E-08 & 1.95E-08 & 4.91E-09 & 1.25E-09 & 3.14E-10\\
  &  &  0.51 & 1.79 & 1.99 & 1.97 & 2.00   \\ \hline
 $10^{-7}$ & 8.45E-09 & 8.03E-09 & 2.69E-09 & 6.80E-10 & 1.70E-10 & 4.24E-11\\
  &  &  0.07 & 1.58 & 1.98 & 2.00 & 2.01  \\ \hline
$10^{-8}$ & 7.01E-10 & 9.00E-10 & 3.49E-10 & 8.86E-11 & 2.23E-11 & 5.76E-12\\
 &  &  * & 1.37 & 1.98 & 1.99 & 1.95   \\ \hline
 $10^{-9}$ & 5.54E-11 & 9.62E-11 & 4.32E-11 & 1.10E-11 & 2.86E-12 & 5.74E-13 \\
  &  &  * & 1.15 & 1.97 & 1.95 & 2.31 
\end{tabular}
	\caption{The maximum errors $E^N_\eps$ and the corresponding convergence rates $r^N_\eps$ on $S_{1/\eps}$ for the problem 
\eqref{ex2}. ($^*$Convergence cannot be observed.)}
 	\label{tab:ex2_SE}
\end{table}

\section{Conclusion}
\label{sec:concl}

When the reduced solution $u_0$ is found exactly or approximately with sufficient accuracy, the decomposition 
$u=u_0+w$ can be used to transform the singularly perturbed convection-diffusion problem \eqref{1DCD} so that $w$ 
is the new unknown function. Since there is one less component to calculate, this can improve the results of any 
discretization scheme, but the order of accuracy generally remains unchanged. However, for the Samarskii-type 
schemes on piecewise uniform meshes of Shishkin type, the order of accuracy increases from 1 to almost 2, \cite{VNfc}. 
For one of those schemes, viz., the exponentially fitted Allen-Southwell-Il'in (ASI) scheme, the improvement is even 
more significant. This is established in the present paper both theoretically and numerically on the original Shishkin 
mesh and its modification that replaces the standard transition point of the form $\Oh(\eps\ln N)$ with
$\Oh(\eps |\ln \eps|)$. In some cases, the errors not only diminish when the number of mesh steps $N$ increases, but 
also when $\eps$ decreases. The decomposition is not needed when $u_0$ is linear.

Although the ASI scheme was introduced a long time ago, it has not been analyzed on layer-adapted nonuniform 
meshes, except as a member of a general class of Samarskii-type or similar schemes, \cite{Linss02, VNfc}. We 
fill this gap here, at least for the Shishkin-type meshes. 

\end{document}